\documentclass[11pt, a4paper]{article} 

\pdfoutput=1

\usepackage[hmargin=1.5cm, vmargin={1.75cm,1.75cm}]{geometry}   

\usepackage[american]{babel}

\usepackage{url}                                

\usepackage{abstract}

\usepackage{appendix}

\usepackage[nottoc]{tocbibind}

\usepackage{savesym}          

\savesymbol{iint}                      
\savesymbol{iiint}                     
\savesymbol{oiint}                    
\savesymbol{Square}              

\usepackage{amsmath}          
\usepackage[matha]{mathabx}
\usepackage{variantbbm} 
\usepackage{variantrsfs} 
\usepackage{eufrak} 

\usepackage[thmmarks, amsmath, thref, hyperref]{ntheorem}

\usepackage{amsxtra}           

\usepackage{extarrows}         

\usepackage{chemarr}           

\usepackage{leftidx}                

\usepackage{ulsy}                   

\usepackage{tensor}               

\usepackage[pdf,all]{xy}         
\savesymbol{merge}               
\savesymbol{boxdot}              

\usepackage{stmaryrd}          
\restoresymbol{stmaryrd}{merge}
\restoresymbol{stmaryrd}{boxdot}

\usepackage{cancel}              

\usepackage{bm}                    

\usepackage{array}                
\usepackage{dcolumn}         
\usepackage{longtable}        
\usepackage{hhline}              

\usepackage{textcomp}          

\usepackage{eurosym}           

\usepackage{keyval}               
\usepackage{graphics}           

\usepackage{caption}		

\ifpdf
\usepackage[pdftex, hiresbb]{graphicx}
\else
\usepackage[hiresbb]{graphicx}
\fi

\ifpdf
\DeclareGraphicsExtensions{.pdf, .jpg, .tif}
\else
\DeclareGraphicsExtensions{.eps, .jpg}
\fi

\def\xLeftarrowfill@{\arrowfill@\Leftarrow\Relbar\Relbar}
\newcommand{\xLeftarrow}[2][]{%
    \ext@arrow 0099\xLeftarrowfill@{#1}{#2}} 

\def\xRightarrowfill@{\arrowfill@\Relbar\Relbar\Rightarrow}
\newcommand{\xRightarrow}[2][]{%
    \ext@arrow 0099\xRightarrowfill@{#1}{#2}} 

\newcommand{\Beta}{\mathrm{B}}

\newcommand{\Chi}{\mathrm{X}}

\theoremstyle{break}
\theoremheaderfont{\normalfont\bfseries}
\theorembodyfont{\itshape}
\theoremseparator{:}
\theoremnumbering{arabic}
\theoremsymbol{\$}

\newtheorem{proposition}{Proposition}[section]

\newtheorem{lemma}{Lemma}[section]

\theoremstyle{nonumberbreak}

\theorembodyfont{\upshape}
\theoremstyle{nonumberplain}
\newtheorem{proof}{Proof}

\theoremsymbol{}
\newtheorem{definition}{Definition}

\newcommand*{\can}{\ensuremath{\mathrm{can}}}
\newcommand*{\Bor}{\ensuremath{\mathrm{Bor}}}
\newcommand*{\LSpace}{\ensuremath{\mathrm{L}}}
\newcommand*{\ESpace}{\ensuremath{\mathrm{E}}}
\newcommand*{\Dirichlet}{\ensuremath{\mathrm{Dir}}}

\title{Explicit measures for the homogeneous transform}

\author{Hubert Holin\\
\url{Hubert.Holin.1982@Polytechnique.org}}

\date{September 1, 2016}							

\begin{document}

\maketitle

\begin{abstract}
The homogeneous transform has many practical applications outside the realm of mathematics,
for instance to represent the proportions of several chemical substances. We aim here to present
results about the transformation of measures, which could be used to take into account the uncertainties
of the quantities to be homogeneously transformed.
\end{abstract}

\tableofcontents
\pagebreak

\section{Basics of the homogeneous transform}\label{simppres}

\subsection{Building blocks}\label{basbuiblo}
\indent

Let $n \in \bbN$, $n \neq 0$, we will assimilate $\bbR^{n}$ with
$\bbR^{n} \times \{0\}$ as a subset of $\bbR^{n+1}$.

Given $q$ linearly independent points in an affine space, $(M_{1}, \dots, M_{q})$, recall that the
geometric simplex spawned by $\{M_{1}, \dots, M_{q}\}$ is the set of points whose barycentric
coordinates with respect to $(M_{1}, \dots, M_{q})$ are positive, or equivalently, are positive and less than or equal to 1; it coïncides with the convex hull of $\{M_{1}, \dots, M_{q}\}$. Likewise, the affine
space spawned by $\{M_{1}, \dots, M_{q}\}$ is the set of points whose barycentric coordinates with
respect to $(M_{1}, \dots, M_{q})$ are unconstrained.

For $ i \in \{1, \ldots, n+1\} $ let $ A_{i} = (\delta_{i,1}, \ldots , \delta_{i,n+1}) \in\bbR^{n+1} $ with
$ \delta_{i,j} = 1 $ if $ i = j $ and $ \delta_{i,j} = 0 $ otherwise, and let $O = (0, \ldots , 0)$.
Recall (\cite{GreenbergHarper(1981)}) that the  \emph{standard $n$-dimensional geometric simplex} is
the $n$-dimensional geometric simplex spawned by $\{O, A_{1}, \ldots , A_{n}\}$, with $O = (0, \ldots , 0)$,
while the \emph{probability simplex} (sometimes also known as the \emph{standard $n$-dimensional simplex}) is the
$n$-dimensional geometric simplex spawned by $(A_{1}, \ldots , A_{n+1})$.

Let $B_{n} = \{(x_{1}, \ldots, x_{n}) \in\  ]\  0 ; 1 \ [^{n} \ \vert\  \sum_{i=1}^{n} x_{i} < 1\} \subset \bbR^{n}$.
Let $P_{n}$ be the affine hyperplane of
$\bbR^{n+1}$ spawned by $\{A_{i}\  \vert\  i \in \{1, \ldots, n+1\}\}$ and
$C_{n} = \{\sum_{i = 1}^{n+1} y_{i} A_{i}\ \vert (\forall i \in \{ 1, \ldots, n+1 \})\  0 < y_{i} < 1 \}$.
Note that the coordinates of any point in $C_{n}$ coïncide with its barycentric coordinates relative to
$(A_{1}, \ldots , A_{n+1})$, and that $C_{n}$ is the interior both in the sense of the usual topology on
$\bbR^{n}$ and in the sense of manifolds with boundary, of the probability simplex; likewise,
$B_{n}$ is the interior, both in the sense of the usual topology on $\bbR^{n}$ and in the sense of manifolds with boundary, of the standard n-dimensional geometric simplex. Figure~\ref{simplices} illustrates the difference between
$B_{n}$ and $C_{n}$. Finally, 
let $U_{n+1} = \{(y_{1}, \ldots, y_{n+1}) \in \bbR^{n+1}\ \vert \sum_{i = 1}^{n+1} y_{i} > 0\}$;
quite clearly $P_{n} \subset U_{n+1} $.

\begin{figure}[h]
\centering
\begin{minipage}{.5\textwidth}
\raggedleft
\includegraphics[scale=0.25]{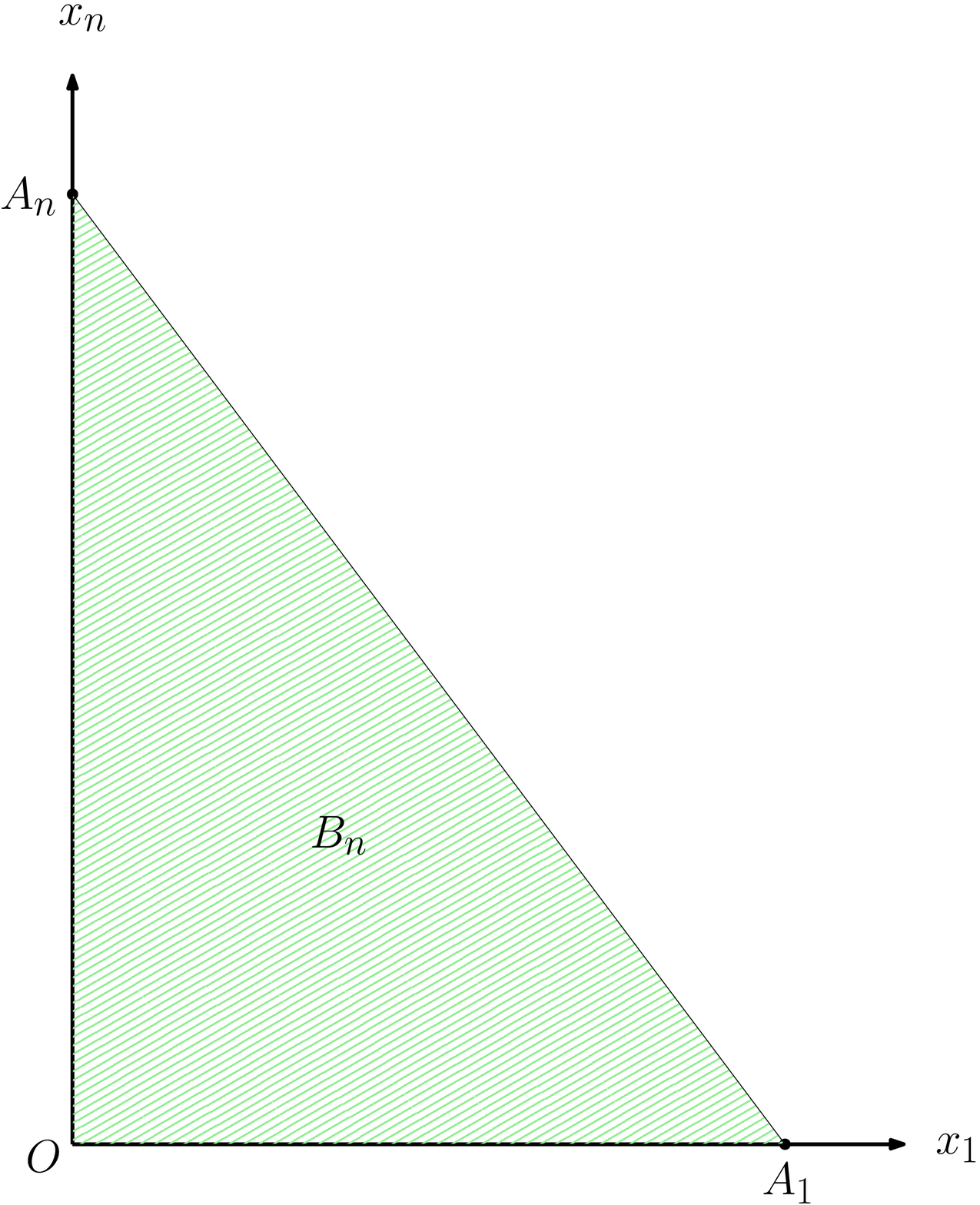}
\end{minipage}%
\begin{minipage}{.5\textwidth}
\raggedright
\includegraphics[scale=0.25]{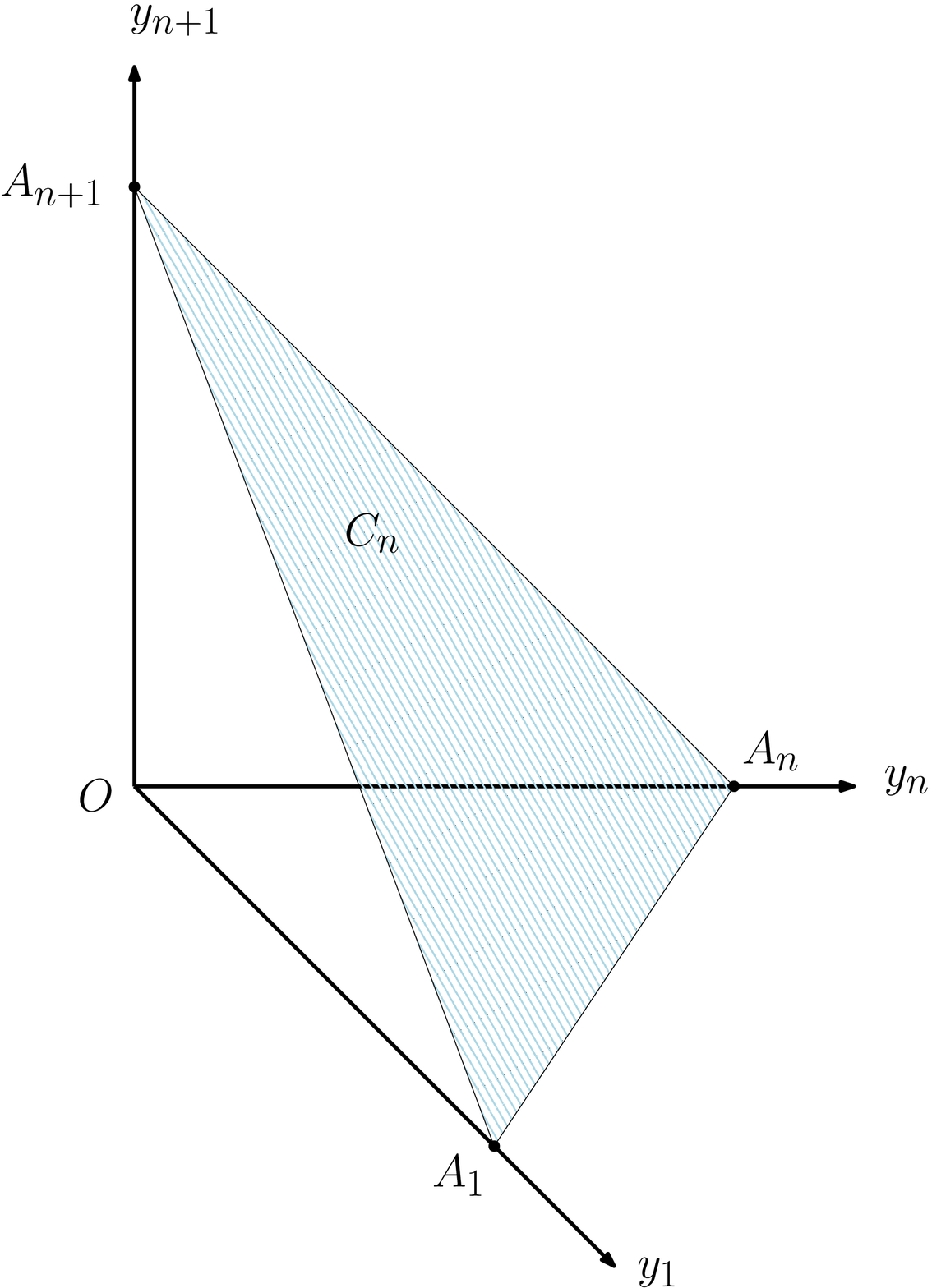}\end{minipage}
\caption{The standard $n$-dimensional geometric simplex (left) and the probability simplex (right).}\label{simplices}
\end{figure}

We now define the \emph{homogeneous transform} as:
\[
\begin{array}{rccc}
\rsfsH_{n+1} \negthinspace : & \ U_{n+1} & \rightarrow & \ P_{n} \\
 & (y_{1}, \ldots , y_{n+1}) & \mapsto & {\displaystyle
 \left(\frac{y_{1}}{ \sum_{i = 1}^{n+1} y_{i}}, \ldots , \frac{y_{n+1}}{ \sum_{i = 1}^{n+1} y_{i}}\right)}
\end{array}
\]

We note that $\rsfsH_{n+1}\!<U_{n+1}\!>\ = P_{n}$,
$\rsfsH_{n+1}^{-}\!<P_{n}\!>\ = U_{n+1}$,
$\rsfsH_{n+1}\!<(\bbR_{+}^{\ast})^{n+1}\!>\ = C_{n}$
 and \linebreak$\rsfsH_{n+1}^{-}\!<C_{n}\!>\ = (\bbR_{+}^{\ast})^{n+1}$.
 We will denote by the same symbol the function  $\rsfsH_{n+1}$ and its restriction to
 $(\bbR_{+}^{\ast})^{n+1}$ and $C_{n}$.

\subsection{The trivialisation of the homogeneous transform}\label{triho}
\indent

We note that $\rsfsH_{n+1}$ is a (global) fibration (\cite{Leborgne(1982)}), both on $U_{n+1}$
and on $C_{n}$, the fibre being an open half-line: given
$\can : P_{n} \times \bbR_{+}^{\ast} \rightarrow P_{n}\ ;\ (Q,s) \mapsto Q$ and
$\varphi : P_{n} \times \bbR_{+}^{\ast} \rightarrow U_{n+1}\ ;\ (Q,s) \mapsto s.Q$
(this latter being
the component-wise multiplication of the coordinates of $Q$ by $s$, and being a
$C^{\infty}$-diffeomorphism), we clearly have, denoting again by the same symbols $\varphi$
and $\can$ and their restrictions to the sets of interest,
$\rsfsH_{n+1} = \can\ \circ\ \varphi^{-1}$:
\[
\begin{array}{cc}
\xymatrix{
U_{n+1}\ar[d]_{\rsfsH_{n+1}} & P_{n} \times \bbR_{+}^{\ast}\ar[l]_{\varphi}\ar[ld]^{\can}\\
P_{n}\\
} &
\xymatrix{
(\bbR_{+}^{\ast})^{n+1} \ar[d]_{\rsfsH_{n+1}} &
C_{n} \times \bbR_{+}^{\ast}\ar[l]_{\varphi}\ar[ld]^{\can}\\
C_{n}\\
} \\
\end{array}
\]

It is easier to look at this thru a chart; let us thus
consider $\rsfsC_{n} : \bbR^{n} \rightarrow P_{n}\ ;\ (x_{1},\ldots,x_{n}) \mapsto
(x_{1},\ldots,x_{n},\linebreak 1-\sum_{i = 1}^{n} x_{i})$.
We can now build
$\rsfsT_{n+1} : \bbR^{n} \times \bbR \rightarrow
\bbR^{n+1}\ ;\ ((x_{1},\ldots,x_{n}), t) \mapsto \varphi(\rsfsC_{n}(x_{1},\ldots,x_{n}),
\mathrm{e}^{t}) =
(\mathrm{e}^{t}.x_{1},\ldots,\mathrm{e}^{t}.x_{n}, \mathrm{e}^{t}.(1-\sum_{i = 1}^{n} x_{i}))$
and each of the following graphs commutes, with
$\pi_{n+1} :  \bbR^{n} \times \bbR \rightarrow \bbR^{n}$
the canonical projection, and where we have again denoted by the same symbols $\rsfsT_{n+1}$,
$\rsfsH_{n+1}$, $\rsfsC_{n}$ and $\pi_{n+1}$ and their restrictions to the sets of interest:
\[
\begin{array}{cc}
\xymatrix{
\bbR^{n} \times \bbR \ar[d]_{\pi_{n+1}} \ar[r]^{\rsfsT_{n+1}} &
U_{n+1} \ar[d]^{\rsfsH_{n+1}} \\
\bbR^{n} \ar[r]_{\rsfsC_{n}} & P_{n}\\
} &
\xymatrix{
B_{n} \times \bbR \ar[d]_{\pi_{n+1}} \ar[r]^{\rsfsT_{n+1}} &
(\bbR_{+}^{\ast})^{n+1} \ar[d]^{\rsfsH_{n+1}} \\
B_{n} \ar[r]_{\rsfsC_{n}} & C_{n}\\
} \\
\end{array}
\]

$\rsfsT_{n+1}$ and $\rsfsC_{n}$ are obviously $C^{\infty}$-diffeomorphisms, for which
we compute at once the following, with $J_{\rsfsT_{n+1}}$ and $J_{\rsfsC_{n}}$
the jacobian matrices of $\rsfsT_{n+1}$ and $\rsfsC_{n}$ respectively, and
with $I_{n}$ the $(n,n)$ identity matrix:
\[
\begin{array}{ccccc}
J_{\rsfsC_{n}} = \left[
\begin{array}{ccc}
 & I_{n} & \\ \hline
 -1 & \cdots & -1 \\
\end{array}
\right]
 & , &
J_{\rsfsT_{n+1}} = \mathrm{e}^{t} \left[
\begin{array}{c|c}
J_{\rsfsC_{n}} &
\begin{array}{c}
x_{1} \\
\vdots \\
x_{n} \\ \hline
1-(x_{1} + \cdots + x_{n}) \\
\end{array} \\
\end{array}
\right] & , &
\det\!\left(J_{\rsfsT_{n+1}}\right) = \mathrm{e}^{(n+1)\,t} \\
\end{array}
\]

\section{Manifolds and Measures}\label{mamea}
\indent

We first present here some general results before applying them to the homogeneous transform
and its associated (very simple) manifolds in the next section.

\subsection{Measures on a manifold}\label{meame}
\indent

The topic of measures on a manifold is a classic, well-studied one, with numerous variants
(orientable, Riemannian, Lipschitz, infinite-dimensional, \emph{etc.}). For our needs here,
we will first briefly restate the idea for $C^{1}$-submanifolds of $\bbR^{q}$, without border,
merely to make our conventions and assumptions clear (note, for instance, that the charts
in  \cite{Khoan(1972)} are backwards). We will assume nothing more than the $C^{1}$-submanifold
structure, beyond the requirement that $V$ be connected, for simplicity's sake.

We consider on $\bbR^{q}$ ($q \neq 0$),  $\mathcal{T}(\bbR^{q})$ the usual topology
defined by the euclidian metric.  Recall that a set $V \subset \bbR^{q}$ is a $C^{1}$-submanifolds
(of $\bbR^{q}$) without border, of dimension $p$ if and only if for all $M \in V$, there exists an open
neighbourhood $U_{M}$ of $M$ in $\bbR^{q}$ and a $C^{1}$-diffeomorphism $\psi_{M}$ from
$U_{M}$ onto its image $W_{M} \in \mathcal{T}(\bbR^{q})$, such that
$\psi_{M}\!<\!U_{M}\cap V\!>\ = ( \bbR^{p} \times \{0_{\bbR^{q-p}}\}) \cap \psi_{M}\!<\!V\!>$,
with $0_{\bbR^{q-p}}$ being the $0$ of $\bbR^{q-p}$. $\psi_{M}$ is called a local chart
adapted to $M$ and $\psi_{M}^{-1}\vert_{ \bbR^{p}}$ is called a local parametrisation
(it will be advantageous to consider that $\psi_{M}^{-1}\vert_{ \bbR^{p}}$ is a function to
$\bbR^{q}$ and not just to $M$, even though its image is indeed in $M$).
We build an atlas, defining a $C^{1}$-manifold structure (hence called a $C^{1}$-submanifold structure)
for $V$ by collecting
for each $M \in V$ one local chart adapted to $M$. Any other atlas built by choosing for each $M \in V$
another local chart adapted to $M$ is of course compatible) with this one (\emph{i.e.} the transition
functions from one atlas to the other are $C^{1}$-diffeomorphism where their domains overlap.

We now briefly recall the salient points of the  $C^{1}$-submanifold structure we will need further on:

\begin{lemma}\label{topoprop}
If $V$ is a connected $C^{1}$-submanifolds of $\bbR^{q}$ without border, then its topology is
metric, Hausdorff, separable, second countable, locally compact, locally connected and
$\sigma$-compact. Furthermore it is a normal space and for any given open cover of $V$,
there exists a partition of unity admissible to it.
\end{lemma}

\begin{proof}
As the topology of $V$ is deduced from that of $\bbR^{q}$ which is metric,
Hausdorff and separable (\emph{i.e.} there is a countable subset which is everywhere dense),
hence is second countable (\emph{i.e.} thas a countable basis, \cite[IX, page 18]{Bourbaki(TG)}),
it is likewise metric, Hausdorff, separable and second countable.

The manifold construction clearly ensures that the topology of $V$ is locally compact and locally
connected. Taken together, the separability and local compactness implies that $V$
is the union of a countable family of compacts (\emph{i.e.} $V$ is $\sigma$-compact,
\cite[IX, page 21]{Bourbaki(TG)}). Furthermore, $V$ is a normal space as it is metrisable\linebreak
(\cite[IX, page 43]{Bourbaki(TG)}) hence for any given open cover of $V$,
there exists a partition of unity admissible to it \linebreak(\cite[IX, page 47]{Bourbaki(TG)}).
\end{proof}

\begin{lemma}\label{countableatlas}
If $V$ is a connected $C^{1}$-submanifolds of $\bbR^{q}$, without border, whose defining
atlas is $\mathfrak{A}$, then there is a countable atlas $\mathfrak{B}$ defining the same
manifold structure of $V$ as $\mathfrak{A}$ and such that the domain  of every chart in $\mathfrak{B}$
is a countable union of compact sets.
\end{lemma}

\begin{proof}
We follow \cite{Voronov(2009)}. As the topology of $V$ is second-countable, there exists a countable
family $\mathcal{B}$ of open sets of $V$ which is a basis of its topology.
Hence the domain of every chart in $\mathfrak{A}$ is a countable union of elements of $\mathcal{B}$.
For each chart $\mathfrak{c} \in \mathfrak{A}$, let $\mathcal{B}_{\mathfrak{c}}$ be the subset of
$\mathcal{B}$ such that the domain of $\mathfrak{c}$ is the union of the elements of
$\mathcal{B}_{\mathfrak{c}}$; let
$\mathcal{B}_{\mathfrak{A}} = \bigcup_{\mathfrak{c} \in \mathfrak{A}} \mathcal{B}_{\mathfrak{c}}$,
then $\mathcal{B}_{\mathfrak{A}}$ is countable.

For each $B \in \mathcal{B}_{\mathfrak{A}}$, choose one chart $\mathfrak{c}$ of $\mathfrak{A}$ in whose
domain $B$ is included, and consider $\mathfrak{c}_{B}$ the restriction of $\mathfrak{c}$ to $B$.
Then the collection of these $\mathfrak{c}_{B}$ is a countable atlas which, is compatible with
$\mathfrak{A}$, and hence also defines the manifold structure of $V$.
 
 Finally, each domain of our countable atlas is $\sigma$-finite, being homeomorphic to an open subset
 of $\bbR^{p}$.
 \end{proof}

To build the measure we are after, assume first that $V$ admits an atlas with just one chart.
Hence we consider the open cover
of $V$ consisting solely of $V$, and consider the associated parametrisation
\linebreak$\Phi : \Omega \rightarrow V$, with $\Omega$ a non-empty subset of $\bbR^{p}$.
In that case it is possible to define a positive Radon measure
$\tilde{\mu}_{V} : \rsfsK(V) \rightarrow \overline{\bbR}{}^{+}$ by
$\tilde{\mu}_{V}(\varphi) = \int_{\Omega} (\varphi\ \circ\ \Phi) A_{\Phi}(t_{1}, \ldots , t_{p})\ dt_{1} \ldots dt_{p}$
with
\[
A_{\Phi}(t_{1}, \ldots , t_{p}) = \sqrt{\sum_{1 \leqslant j_{1} < \cdots < j_{p} \leqslant q}
\left\vert \frac{\text{D}(\Phi^{j_{1}}, \ldots , \Phi^{j_{p}})}{\text{D}(t_{1}, \ldots , t_{p})} \right\vert^{2}}
\]

The measure $\mu_{V}$ on $\Bor(V)$ associated with $\tilde{\mu}_{V}$ therefore verifies
$\mu_{V} = \Phi\,{}_{\ast}\,(A_{\phi} \lambda_{\bbR^{p}})$, with $\lambda_{\bbR^{p}}$ the Lebesgues
measure on $\bbR^{p}$; $\mu_{V}$
is thus a pushforward of a \emph{weighted} Lebesgues measure,
which reduces to the classical
change of variables expression when $V$ is simply an open subset of $\bbR^{q}$.

When the atlas has more than one chart, the above is performed on each chart domain, using a
partition of unity admissible to the open cover consisting of the domains of the charts in the atlas,
and noticing that the measures agree on the overlaps. Furthermore, and quite
importantly, one verifies that the measure built does not, in fact, depend upon the specific
atlas chosen, so long as it is compatible with the manifold structure

\begin{proposition}\label{sigmafinitude}
The measure $\mu_{V} $ is $\sigma$-finite.
\end{proposition}

\begin{proof}
Given one countable atlas defining the manifold structure of $V$, we choose a countable family of
compacts such that their union is $V$ and that each compact is included in the domain of one chart of the
atlas. Their image by the chart will be a compact of  $\bbR^{p}$, because each chart is an
homeomorphism. As the compacts of $\bbR^{p}$ have finite Lebesgues measure and the parametrisation are $C^{1}$, this proves
that each compact of our chosen family will have finite measure for the measure we have built here.
Hence, our measure on $V$ is $\sigma$-finite.
\end{proof}

Note that for the applications we present here, we will be in a situation where an atlas with
just one chart will suffice.

\subsection{Lebesgues-Radon-Nikod\'{y}m variants}\label{radnykvar}
\indent

We will follow the conventions of \cite{Rudin(1986)} and consider that a \emph{complex} measure
never takes the value $+\infty$. Given a (positive or complex) measure $\lambda$ on some
measurable set and a (positive) measure $\mu$ on that same measurable set, we will indicate
the fact that $\lambda$ is absolutely continuous with respect to $\mu$ by the notation
$\lambda \lll \mu$, and likewise, if $\lambda_{1}$ and $\lambda_{2}$ are two (positive or complex)
measures, we will indicate the fact that $\lambda_{1}$ and $\lambda_{2}$ are mutually singular
by the notation $\lambda_{1} \perp \lambda_{2}$. Given $(X, \rsfsA, \mu)$ a measure space
(with $\mu$ positive) and $\lambda$ a complex measure on  $(X, \rsfsA)$, we will denote by
$(\lambda_{a}, \lambda_{s})$ the Lebesgues decomposition of $\lambda$ relative to $\mu$,
\emph{i.e.} $\lambda_{a}$ and $\lambda_{s}$ are the unique complex measures such that
$\lambda = \lambda_{a}+\lambda_{s}$, $\lambda_{a} \lll \mu$ and $\lambda_{s} \perp \mu$.

Given two measurable spaces $(X_{1}, \rsfsA_{1})$ and $(X_{2}, \rsfsA_{2})$, we will
denote by $\rsfsM((X_{1}, \rsfsA_{1}), (X_{2}, \rsfsA_{2}))$ the set of
measurable (point-wise) functions from $(X_{1}, \rsfsA_{1})$ to $(X_{2}, \rsfsA_{2})$.
If $\mu_{1}$ is a measure on $(X_{1}, \rsfsA_{1})$, we will denote as usual by
 $\Phi\,{}_{\ast}\,\mu_{1}$ the push forward of $\mu_{1}$, by $\Phi$, \emph{i.e.}
 the measure $\mu_{2}$ on $(X_{2}, \rsfsA_{2})$ defined by
 $(\forall E \in \rsfsA_{2})\ \mu_{2}(E) = \mu_{1}(\Phi^{-}\!<E\!>)$.
 
 \begin{definition}
Given  $(X_{1}, \rsfsA_{1}, \mu_{1})$ and $(X_{2}, \rsfsA_{2}, \mu_{2})$ two
measure spaces and $\Phi \in \rsfsM((X_{1}, \rsfsA_{1}), (X_{2}, \rsfsA_{2}))$,
we will say (departing slightly from \cite{Malliavin(1982)}) that $\Phi$ is a \emph{morphism of
measure spaces} if and only if $\Phi\,{}_{\ast}\,\mu_{1} = \mu_{2}$.
\end{definition}

\begin{lemma}\label{elemmorph}
Let $(X_{1}, \rsfsA_{1}, \mu_{1})$ and $(X_{2}, \rsfsA_{2}, \mu_{2})$ be two
measure spaces ($\mu_{1}$ and $\mu_{2}$ positive) and $\Phi: X_{1} \rightarrow X_{2}$ a
morphism of
measure spaces. Let $\lambda$ be a complex measure on $(X_{1}, \rsfsA_{1})$.
\begin{itemize}
\item[a)] If $\lambda \lll \mu_{1}$ then $\Phi\,{}_{\ast}\,\lambda \lll \mu_{2}$.
\item[b)] If $\Phi$ is \emph{injective} and $\lambda \perp \mu_{1}$ then
$\Phi\,{}_{\ast}\,\lambda \perp \mu_{2}$.
\end{itemize}
\end{lemma}

\begin{proof}
Assume that $\lambda \lll \mu_{1}$. Let $Y \in  \rsfsA_{2}$ such that $\mu_{2}(Y) = 0$.
We then have $\mu_{2}(Y) = \mu_{1}(\Phi^{-}\!\!<\!Y\!\!>)$ because $\Phi$ is a morphism and hence
$\mu_{1}(\Phi^{-}\!\!<\!Y\!\!>) = 0$. By definition,
$(\Phi\,{}_{\ast}\,\lambda)(Y) = \lambda(\Phi^{-}\!\!<\!Y\!\!>)$. Since $\mu_{1}(\Phi^{-}\!\!<\!Y\!\!>) = 0$
and $\lambda \lll \mu_{1}$ we find that $\lambda(\Phi^{-}\!\!<\!Y\!\!>) = 0$ which means that
$(\Phi\,{}_{\ast}\,\lambda)(Y) = 0$. This proves the first assertion of the lemma.

Assume that $\lambda \perp \mu_{1}$. Let $V_{1} \in \rsfsA_{1}$ and
$W_{1} \in \rsfsA_{1}$
such that $V_{1} \cap W_{1} = \emptyset$,
$(\forall Z \in \rsfsA_{1})\ \mu_{1}(Z) = \mu_{1}(Z \cap V_{1})$ and
$(\forall Z \in \rsfsA_{1})\ \lambda(Z) = \lambda(Z \cap W_{1})$.
Let $V_{2} = \Phi\!<\!V_{1}\!>$ and $W_{2} = \Phi\!<\!W_{1}\!>$; since $\Phi$ is \emph{injective},
and $V_{1} \cap W_{1} = \emptyset$ we see that $V_{2} \cap W_{2} = \emptyset$.
For any $Y \in  \rsfsA_{2}$ we
note that $\mu_{2}(Y) = \mu_{1}(\Phi^{-}\!(\Phi^{-}\!\!<<\!Y\!\!>) =  \mu_{1}(\Phi^{-}\!\!<\!Y\!\!>\!\cap\ V_{1})$.
Since $\Phi$ is \emph{injective}, $V_{1} = \Phi^{-}\!\!<\!\Phi\!\!<\!\!V_{1}\!\!>> = \Phi^{-}\!\!<\!\!V_{2}\!\!>$,
and thus $\mu_{2}(Y) = \mu_{1}(\Phi^{-}\!\!<\!Y\!\!> \cap\ \Phi^{-}\!\!<\!V_{2}\!>) =
\mu_{1}(\Phi^{-}\!\!<\!Y \cap V_{2}\!>)$, the later equality being due, once again to the injectivity of
$\Phi$ . Since $\mu_{1}(\Phi^{-}\!\!<\!Y \cap V_{2}\!>) = \mu_{2}(Y \cap V_{2})$, we finally find that
$\mu_{2}(Y) = \mu_{2}(Y \cap V_{2})$. Likewise, taking the positive and negative variations of
the real and imaginary parts of $\lambda$,
$(\Phi\,{}_{\ast}\,\lambda)(Y) = (\Phi\,{}_{\ast}\,\lambda)(Y \cap W_{2})$. This proves the second assertion of the lemma.
\end{proof}

The injectivity is sufficient to prove the second assertion above, but is not necessary:
for $\Phi: [0;2[ \rightarrow [0;1[$ such that $\Phi(x) = x-E(x)$, with $E(x)$ the integer part of $x$, and
$\lambda$ the Lebesgues measure on $[0;2[$, we find that
$\Phi\,{}_{\ast}\,\lambda = 2 \lambda$
(denoting by the same symbol $\lambda$ and its restriction to $[0;1[$), and certainly,
$\nu \perp \lambda \Leftrightarrow \nu \perp (2 \lambda)$. However, that assertion can also
be wrong when $\Phi$ is not injective: for  $\Phi: [0;1] \rightarrow [0;1], x \mapsto 1/2$, 
$\lambda$ the Lebesgues measure on $[0;1]$, and $\delta_{1/2}$ the Dirac measure at $1/2$,
we do have $\delta_{1/2} \perp \lambda$, but
$\Phi\,{}_{\ast}\,\lambda\!=\!\Phi\,{}_{\ast}\,\delta_{1/2}\!=\!\delta_{1/2}$!

\begin{proposition}[stability]\label{stamorph}
Let $(X_{1}, \rsfsA_{1}, \mu_{1})$ and $(X_{2}, \rsfsA_{2}, \mu_{2})$ be two
measure spaces ($\mu_{1}$ and $\mu_{2}$ positive) and $\Phi: X_{1} \rightarrow X_{2}$ an
\emph{injective} morphism of measure spaces. Let $\lambda_{1}$ and $\lambda_{2}$ be complex
measures on $(X_{1}, \rsfsA_{1})$ and $(X_{2}, \rsfsA_{2})$ respectively and
$(\lambda_{1,a}, \lambda_{1,s})$ and $(\lambda_{2,a}, \lambda_{2,s})$ the Lebesgues
decompositions of $\lambda_{1}$ with respect to $\mu_{1}$ and $\lambda_{2}$ with respect to
$\mu_{2}$ respectively. If $\Phi\,{}_{\ast}\,\lambda_{1} = \lambda_{2}$,
then $\Phi\,{}_{\ast}\,\lambda_{1,a} = \lambda_{2,a}$ and
$\Phi\,{}_{\ast}\,\lambda_{1,s} = \lambda_{2,s}$.
\end{proposition}

\begin{proof}
We find immediately that
$\Phi\,{}_{\ast}\,\lambda_{1} = \Phi\,{}_{\ast}\,\lambda_{1,a} + \Phi\,{}_{\ast}\,\lambda_{1,s} = 
\lambda_{2,a} + \lambda_{2,b} $, Lemma~\ref{elemmorph} tels us that\linebreak
$\Phi\,{}_{\ast}\,\lambda_{1,a} \lll \lambda_{2}$, $\Phi\,{}_{\ast}\,\lambda_{1,s} \perp \lambda_{2}$, 
and using the fact that the Lebesgues decomposition of a measure is unique,
we have proved the assertion.
\end{proof}

\begin{lemma}\label{metromorph}
Let $(X_{1}, \rsfsA_{1}, \mu_{1})$ and $(X_{2}, \rsfsA_{2}, \mu_{2})$ be two
measure spaces ($\mu_{1}$ and $\mu_{2}$ positive) and $\Phi: X_{1} \rightarrow X_{2}$ a
morphism of measure spaces. If $\Phi$ is \emph{bijective} and $\Phi^{-1}$ is measurable,
then $\Phi^{-1}$ is also a morphism of measure spaces.
\end{lemma}

\begin{definition}
A function verifying the hypotheses of lemma~\ref{metromorph} will be called a \emph{metromorphism}.
\end{definition}

\begin{proof}
For any $E_{1} \in \rsfsA_{1}$ consider $E_{2} = \Phi\!<\!E_{1}\!> = (\Phi^{-1})^{-}\!<\!E_{1}\!>$;
since $\Phi^{-1}$ is measurable we have $E_{2} \in \rsfsA_{2}$.
Since $\Phi$ is a morphism of measure spaces,
$\mu_{2}(E_{2}) = \mu_{1}(\Phi^{-}\!\!<\!E_{2}\!>) = \mu_{1}(\Phi^{-}\!\!<\!\Phi\!<\!E_{1}\!>\!>) = \mu_{1}(E_{1})$,
but $E_{2} = (\Phi^{-1})^{-}\!<\!E_{1}\!>$ so $\mu_{1}(E_{1}) = \mu_{2}((\Phi^{-1})^{-}\!<\!E_{1}\!>)$.
Finally $\mu_{1} = (\Phi^{-1})\,{}_{\ast}\,\mu_{2}$.
\end{proof}

Given a measure space $(X, \rsfsA, \mu)$, we will denote by $\rsfsL^{0}(\mu)$ the set
$\rsfsM((X_{1}, \rsfsA_{1}), (\bbR, \Bor(\bbR)))$, with
$\Bor(\bbR)$ the borelians of $\bbR$, by $\LSpace^{0}(\mu)$ the set of classes of
elements of $\rsfsL^{0}(\mu)$ with are $\mu$-a.e. equal, by $\rsfsE^{0}(\mu)$ the set of
elementary functions, \emph{i.e.} the set of functions which are finite linear combinations of indicator
functions of $\mu$-measurable sets, by $\ESpace^{0}(\mu)$ the set of classes of such functions,
and by $\ESpace^{1}(\mu)$ the set of \linebreak$\mu$-integrable (classes of)
elementary functions. Recall that $\LSpace^{1}(\mu)$ is a Banach space, and that 
$\ESpace^{1}(\mu)$ is an everywhere dense vector subspace (this latter being a consequence of 
$\LSpace^{1}(\mu)$ being the completion of $\ESpace^{1}(\mu)$ for the $\LSpace^{1}$ norm).

\begin{lemma}\label{metroiso}
Let $(X_{1}, \rsfsA_{1}, \mu_{1})$ and $(X_{2}, \rsfsA_{2}, \mu_{2})$ be two
measure spaces, with $\mu_{1}$ and $\mu_{2}$ positive, 
and $\Phi: X_{1} \rightarrow X_{2}$ a metromorphism between the two. Then
given $H \in \LSpace^{0}(\mu_{1})$, $H \circ \Phi^{-1}$ is well defined,
$(\forall h_{1} \in \ESpace^{1}(\mu_{1}))\linebreak h_{2} = h_{1} \circ \Phi^{-1} \in \ESpace^{1}(\mu_{2})$
and $[ \ESpace^{1}(\mu_{1}) \rightarrow  \ESpace^{1}(\mu_{2}); h_{1} \mapsto h_{1} \circ \Phi^{-1}]$
is (a linear operator which is) an isometry for the $\LSpace^{1}$ norms.
\end{lemma}

\begin{proof}
Let $H \in \LSpace^{0}(\mu_{1})$, and let $f : X_{1} \rightarrow \bbR$ and
$g : X_{1} \rightarrow \bbR$ be two point-wise representatives of $H$. Let $Z$ be
the (possibly
empty) set $\{w \in X_{1} \vert f(w) \neq g(w)\}$. By definition, $\mu_{1}(Z) = 0$. Since $\Phi$ is
a metromorphism, $\mu_{2}(\Phi\!<\!Z\!>) = \mu_{1}(Z) = 0$. But $\Phi\!<\!Z\!>$ is precisely
the (possibly empty) set of points in $X_{2}$ where $f  \circ \Phi^{-1}$ and $g  \circ \Phi^{-1}$ differ.
Finally, $f  \circ \Phi^{-1}$ and $g  \circ \Phi^{-1}$ are $\mu_{2}$-a.e. equal, and $H \circ \Phi^{-1}$ is well defined.

Let $k \in \ESpace^{1}(\mu_{1})$, $k = \sum_{j = 0}^{N} \alpha_{j} \bbone_{W_{j}}$
for some $N \in \bbN$, with the $\alpha_{j} \in \bbC$ and $W_{j} \in  \rsfsA_{1}$,
$\mu_{1}(W_{j})$ finite. We compute at once that
$k \circ \Phi^{-1} =  \sum_{j = 0}^{N} \alpha_{j} \bbone_{\Phi\!<\!W_{j}\!>}$.
Since $\Phi$ is a metromorphism, Lemma~\ref{metromorph} assures us that
$\Phi\!<\!W_{j}\!> \in \rsfsA_{1}$ and thus that $k \circ \Phi^{-1} \in \LSpace^{0}(\mu_{2})$.
Furthermore, $\Phi^{-}\!\!<\!\Phi\!<\!W_{j}\!>>\ = W_{j}$ so\linebreak
$\mu_{2}(\Phi\!<\!W_{j}\!>) = \mu_{1}(\Phi^{-}\!\!<\!\Phi\!<\!W_{j}\!>>) = \mu_{1}(W_{j})$ and
$\mu_{2}(\Phi\!<\!W_{j}\!>)$ is therefore finite. Finally, $k \circ \Phi^{-1} \in \ESpace^{1}(\mu_{2})$
and $\Vert k \Vert_{\LSpace^{1}(\mu_{1})} = \Vert k \circ \Phi^{-1}  \Vert_{\LSpace^{1}(\mu_{2})}$.
\end{proof}

\begin{lemma}[isometry]\label{metrometry}
Let $(X_{1}, \rsfsA_{1}, \mu_{1})$ and $(X_{2}, \rsfsA_{2}, \mu_{2})$ be two
measure spaces, with $\mu_{1}$ and $\mu_{2}$ positive, 
and $\Phi: X_{1} \rightarrow X_{2}$ a metromorphism between the two. Then
$(\forall h_{1} \in \LSpace^{1}(\mu_{1}))\ h_{2} = h_{1} \circ \Phi^{-1} \in \LSpace^{1}(\mu_{2})$
and $[ \LSpace^{1}(\mu_{1}) \rightarrow  \LSpace^{1}(\mu_{2});\linebreak h_{1} \mapsto h_{1} \circ \Phi^{-1}]$
is an isometry for the $\LSpace^{1}$ norms.
\end{lemma}

\begin{proof}
Given Lemma~\ref{metroiso}, the operator $\Psi =
[ \ESpace^{1}(\mu_{1}) \rightarrow  \ESpace^{1}(\mu_{2}); h_{1} \mapsto h_{1} \circ \Phi^{-1}]$ has
a prolongation to $\LSpace^{1}(\mu_{1})$, with image in $\LSpace^{1}(\mu_{2})$, which is also
an isometry and that we will denote by $\Psi$ as well.

Let $H \in \LSpace^{1}(\mu_{1})$, and h a point-wise representative of $H$.
Since $\ESpace^{1}(\mu_{1})$
is dense in $\LSpace^{1}(\mu)$, we can find a sequence
$(K_{i})_{i \in  \bbN} \in (\ESpace^{1}(\mu_{1}))^{ \bbN}$ which converges to $H$
for the $L^{1}$-norm. Let us now consider $(k_{i})_{i \in  \bbN}$,
where the $k_{i}$ are point-wise representatives of the $K_{i}$.
By extracting, if need be,
a subsequence, we may assume that $(k_{i})_{i \in  \bbN}$ converges to $h$ $\mu_{1}$-a.e..
Let $Z$ be the (perhaps empty) set of points of $X_{1}$ where $(k_{i})_{i \in  \bbN}$
does \emph{not} converges to $h$. Then if $y \in X_{2}$ and $y \notin \Phi\!<\!Z\!>$ we have
the convergence of $(k \circ \Phi^{-1})(y)$ to $(h \circ \Phi^{-1})(y)$. Since $\Phi$ is a
metromorphism, $\mu_{2}(\Phi\!<\!Z\!>) = \mu_{1}(Z) = 0$. Therefore, $k \circ \Phi^{-1}$
converges to $h \circ \Phi^{-1}$ $\mu_{2}$-a.e.. Since  $(K_{i})_{i \in  \bbN}$ converges for
the $L^{1}(\mu_{1})$ norm, it is a Cauchy sequence in $\ESpace^{1}(\mu_{1})$ for that norm, and
since $\Psi$ is an isometry from $\ESpace^{1}(\mu_{1})$ to $\ESpace^{1}(\mu_{2})$,
$(K_{i} \circ \Phi^{-1})_{i \in  \bbN}$ is a Cauchy sequence in $\ESpace^{1}(\mu_{2})$ for
the $L^{1}(\mu_{2})$ norm. Therefore, $(K_{i} \circ \Phi^{-1})_{i \in  \bbN}$ converges
to some $F \in \LSpace^{1}(\mu_{2})$. Once again taking a subsequence, if need be, we may
assume that $(k_{i} \circ \Phi^{-1})_{i \in  \bbN}$ converges $\mu_{2}$-a.e. to $f$, where
$f$ is a point-wise representative of $F$.
But then $\mu_{2}$-a.e. $f = h \circ \Phi^{-1}$ and thus
$H  \circ \Phi^{-1} = F \in \LSpace^{1}(\mu_{2})$.

Finally, the linear isometry $\Psi$ is such that
$(\forall H \in\LSpace^{1}(\mu_{1}))\ \Psi(H) = H \circ \Phi^{-1}$.
\end{proof}

Recall that the Radon-Nikod\'{y}m theorem ensures, that if $\mu$ is a positive,
$\sigma$-finite, measure
on a measurable space $(X, \rsfsA)$ and $\lambda$ is a complex measure on
$(X, \rsfsA)$, absolutely continuous with respect to $\mu$, then there is a unique
$h \in \LSpace^{1}(\mu)$ (the Radon-Nikod\'{y}m derivative of $\lambda$ with respect to $\mu$)
such that $(\forall E \in \rsfsA)\ \lambda(E) = \int_{E}h\ d\mu$.

\begin{proposition}[densities]\label{densimorph}
Let $(X_{1}, \rsfsA_{1}, \mu_{1})$ and $(X_{2}, \rsfsA_{2}, \mu_{2})$ be two
measure spaces, with $\mu_{1}$ and $\mu_{2}$ positive and $\sigma$-finite, 
and \linebreak$\Phi: X_{1} \rightarrow X_{2}$
a metromorphism between the two. Let $\lambda_{1}$ be
a complex measure on $(X_{1}, \rsfsA_{1})$  absolutely continuous with respect to $\mu_{1}$,
and $h_{1}$ the Radon-Nikod\'{y}m derivative of $\lambda_{1}$ with respect to $\mu_{1}$. If
$\lambda_{2} = \Phi\,{}_{\ast}\,\lambda_{1}$ then $\lambda_{2}$ is absolutely continuous with respect
to $\mu_{2}$, and its Radon-Nikod\'{y}m derivative  with respect to $\mu_{2}$ is
$h_{2} = h_{1} \circ \Phi^{-1}$.
\end{proposition}

\begin{proof}
The fact that $\lambda_{2} \lll \mu_{2}$ results immediately from Lemma~\ref{elemmorph}.

Let $E \in \rsfsA_{2}$. By definition of the pushforward,
$\lambda_{2}(E) = (\Phi\,{}_{\ast}\,\lambda_{1})(E) = \lambda_{1}(\Phi^{-}\!<E\!>)$, and
by our hypothesis, $ \lambda_{1}(\Phi^{-}\!<E\!>) =  \int_{\Phi^{-}\!<E\!>} h_{1}\ d\mu_{1}$.
Since $\ESpace^{1}(\mu_{1})$
is dense in $\LSpace^{1}(\mu)$, we can find a sequence
$(k_{i})_{i \in  \bbN} \in (\ESpace^{1}(\mu_{1}))^{ \bbN}$ which converges to $h_{1}$
for the $L^{1}$-norm, and then of course $\int_{\Phi^{-}\!<E\!>} k_{i}\ d\mu_{1}$
converges to $\int_{\Phi^{-}\!<E\!>} h_{1}\ d\mu_{1}$.

For any $i_{0} \in  \bbN$, we have
$k_{i_{0}} = \sum_{j = 0}^{N_{i_{0}}} \alpha_{j} \bbone_{W_{j}}$
for some $N_{i_{0}} \in \bbN$, with the $\alpha_{j} \in \bbC$ and
$W_{j} \in  \rsfsA_{1}$. We see that
$\int_{\Phi^{-}\!<E\!>} k_{i_{0}}\ d\mu_{1} =
\sum_{j = 0}^{N_{i_{0}}} \alpha_{j} \int_{\Phi^{-}\!<E\!>} \bbone_{W_{j}}\ d\mu_{1} =
\sum_{j = 0}^{N_{i_{0}}} \alpha_{j} \int_{\Phi^{-}\!<E\!> \cap W_{j}}\ d\mu_{1}$.
Since we have assumed\linebreak 
that $\Phi$ is a metromorphism, it is in particular injective, so
$W_{j} = \Phi^{-}\!\!<\!\Phi\!\!<W_{j}\!>\!> $ and thus\linebreak
$\Phi^{-}\!\!<\!E\!> \cap W_{j} = \Phi^{-}\!\!<E \cap \Phi\!<W_{j}\!>\!>$. Furthermore, being a
metromorphism also implies (because a requirement is that $\Phi^{-1}$ is measurable) that
$\Phi\!<W_{j}\!> \in \rsfsA_{2}$. We therefore find that
$\int_{\Phi^{-}\!<E\!>} k_{i_{0}}\ d\mu_{1} =
\sum_{j = 0}^{N_{i_{0}}} \alpha_{j} \mu_{2}(E \cap \Phi\!<W_{j}\!>) =
\sum_{j = 0}^{N_{i_{0}}} \alpha_{j} \int_{E} \bbone_{ \Phi\!<W_{j}\!>}\ d\mu_{2} =
\int_{E} l_{j_{0}}\ d\mu_{2}$, with
$l_{j_{0}} = \sum_{j = 0}^{N_{i_{0}}} \alpha_{j}  \bbone_{\Phi\!<W_{j}\!>}$ or, in other words,
$l_{j_{0}} = k_{j_{0}} \circ \Phi^{-1}$.
The function $l_{j_0}$ takes a finite number of values, and is a linear combination of indicator
functions of measurable sets; it is therefore an elementary function.
Since $\mu_{2}(\Phi\!<W_{j}\!>) = \mu_{1}(W_{j})$ which is finite, we find that in fact
$l_{j_{0}} \in \ESpace^{1}(\mu_{2})$.

Using Lemma~\ref{metrometry}, $(l_{i})_{i \in  \bbN}$ converges to $h_{1} \circ \Phi^{-1}$ 
in $\LSpace^{1}(\mu_{2})$ and therefore $\int_{\Phi^{-}\!<E\!>} k_{i_{0}}\ d\mu_{1}$ converges
to $\int_{E}(h_{1} \circ \Phi^{-1})\ d\mu_{2}$.

Finally, $\lambda_{2}(E) = \int_{E}\ (h \circ \Phi^{-1})\ d\mu_{2}$ and thus
$h_{2} = h_{1} \circ \Phi^{-1}$.
\end{proof}

When $X_{1}$ and $X_{2}$ are open subsets of some  $\bbR^{n}$, $\mu_{1}$ and $\mu_{2}$
are the Lebesgues measure and $\Phi$ is a $C^{1}$-diffeomorphism, this result is classical
(\cite{MetivierNeveu(1983)}).

\subsection{A note about usage}\label{notus}
\indent

For the use at hand, we are given two measure spaces
$(X, \rsfsA, \mu)$ and $(Y, \rsfsB, \nu)$, with $\mu$ and $\nu$ positive, a function
$H \in \rsfsM((X, \rsfsA), (Y, \rsfsB))$. We want to compute,
for various complex measures $\lambda$ on $(X, \rsfsA)$,
$\rho = H\,{}_{\ast}\,\lambda$, and notably  $(\rho_{a},\rho_{s})$, the Lebesgues decomposition of
$\rho$ with respect to $\nu$ .

In our particular context, we are also given two additional measure spaces
$(S, \rsfsC, \sigma)$
and $(T, \rsfsD, \tau)$, with $\sigma$ and $\tau$ positive, a function
$G \in \rsfsM((S, \rsfsC), (T, \rsfsD))$, and
$\Phi: S \rightarrow X$ and $\Psi: T \rightarrow Y$  two metromorphisms such that the following
graph commutes:
\[
\xymatrix{
(S, \rsfsC, \sigma) \ar[d]_{G} \ar[r]^{\Phi} & (X, \rsfsA, \mu) \ar[d]^{H} \\
(T, \rsfsD, \tau) \ar[r]_{\Psi} & (Y, \rsfsB, \nu)\\
}
\]

Let $\eta =  (G\ \circ\ \Phi^{-1})\,{}_{\ast}\,\lambda = G\,{}_{\ast}\,(\Phi^{-1}\,{}_{\ast}\,\lambda)$.
We find at once that $\rho =  \Psi\,{}_{\ast}\,\eta$, and using the results proved above, we find that
$\rho_{s} =  \Psi\,{}_{\ast}\,\eta_{s}$ and $\rho_{a} =  \Psi\,{}_{\ast}\,\eta_{a}$, with
$(\eta_{a},\eta_{s})$, the Lebesgues decomposition of $\eta$ with respect to $\tau$:
\[
\xymatrix{
(S, \rsfsC)\ \ \ \ \ \ \ar[d]_{G} \ar[r]^{\Phi} & (X, \rsfsA): \lambda \ar[d]^{H} \\
(T, \rsfsD): \eta \ar[r]_{\Psi} & (Y, \rsfsB): \rho\\
}
\]

While we replace one set of problems (computing $\rho$, $\rho_{a}$ and $\rho_{s}$) with a larger set of
problems (computing $\eta$, $\eta_{a}$, $\eta_{s}$, $\Psi\,{}_{\ast}\,\eta_{a}$ and
$\Psi\,{}_{\ast}\,\eta_{s}$), we will
only do this when the resulting problems are each more tractable
than our original ones. Actually, in practice, we will usually be satisfied with just computing
$\eta_{a}$ and $\eta_{s}$, as we will see presently.

\section{The homogeneous transform revisited}

\subsection{Framework}
\indent

We can now put the objects and results of Section \ref{simppres} into the
perspective afforded by Section~\ref{mamea}.

The graphs of subsection~\ref{triho} fit into the usage pattern delineated in subsection~ \ref{notus};
this is outlined for one of them, the substitutions necessary for the other being obvious.

With the notations of \ref{notus}, we start by taking $X = U_{n+1}$, $Y = P_{n}$ and
$H = \rsfsH_{n+1}$. As well, we will choose $S = \bbR^{n} \times \bbR$,
$T = \bbR^{n} $ and $G = \pi_{n+1}$, $\Phi = \rsfsT_{n+1}$ and $\Psi = \rsfsC_{n}$.

Ultimately, we will compute measures on $Y$, so our choices
will be guided from that end. As stated earlier, we are in a situation in which $Y$ is a (trivial) manifold
completely described by one parametrisation, so it will be interesting to have as a reference measure
on $Y$ the manifold measure derived from the Lebesgues measure as described in
subsection~\ref{meame}, which we will denote by $\mu_{P_{n}}$ in accordance with our earlier
convention; this also entails that we choose  $\rsfsB = \Bor(P_{n})$, with $\Bor(P_{n})$ the borelian
sets of $P_{n}$.

As we want $\Psi$ to be a metromorphism, we will take
$\rsfsD =  \Bor(\bbR^{n})$, with $\Bor(\bbR^{n})$ the borelian sets of $\bbR^{n}$,
and $\tau = A_{ \rsfsC_{n}} \lambda_{\bbR^{n}}$, with $\lambda_{\bbR^{n}}$ the
Lebesgues measure on $\bbR^{n}$. Hence
$\mu_{P_{n}} = \rsfsC_{n}\,{}_{\ast}\,(A_{ \rsfsC_{n}} \lambda_{\bbR^{n}})$.

Finally, we will be considering traditional measures on $U_{n+1}$ to transform,
but we also want $\Phi$ to be a metromorphism, so we will take
$\rsfsC = \Bor(\bbR^{n} \times \bbR)$, with
$\Bor(\bbR^{n} \times \bbR)$ the borelian sets of
$\bbR^{n} \times \bbR$ , $\rsfsA = \Bor(U_{n+1})$, with $\Bor(U_{n+1})$ the
 borelian sets of $U_{n+1}$, but for convenience's sake, we will take $\sigma = \lambda_{\bbR^{n+1}}$, the
Lebesgues measure on $\bbR^{n+1}$ and
$\mu = \Phi\,{}_{\ast}\,\sigma = \rsfsT_{n+1}\,{}_{\ast}\,\lambda_{\bbR^{n+1}}$; we will also denote
$\mu$ as chosen here by $\mu_{U_{n+1}}$, which is also coherent with the notations of
subsection~\ref{meame}.

Using subsection~\ref{triho} we see that $\det\!\left(J_{\rsfsT_{n+1}}\right) = \mathrm{e}^{(n+1)\,t} =
\left(\sum_{i = 1}^{n+1} y_{i} \right)^{(n+1)}$ with
$(y_{1},\cdots,y_{n+1}) = \linebreak\rsfsT_{n+1}((x_{1},\cdots,x_{n}), t)$ and therefore
$\mu = \left\vert\,\det\!\left(J_{\rsfsT_{n+1}}\right)\right\vert^{-1}\  \lambda_{\bbR^{n+1}} =
\left(\sum_{i = 1}^{n+1} y_{i} \right)^{-(n+1)}\ \lambda_{\bbR^{n+1}}$ and subsection~\ref{meame}
gives us $A_{ \rsfsC_{n}} = \sqrt{n+1}$.
 We can therefore annotate thus the graph of subsection~\ref{notus}:
\[
\xymatrix{
*+[l]{(\bbR^{n} \times \bbR, \lambda_{\bbR^{n+1}}) =
(S, \sigma)} \ar[d]_{\pi_{n+1} = G}
\ar[r]^{\Phi = \rsfsT_{n+1}} &
*+[r]{(X, \mu) =
\left(U_{n+1}, \left(\sum_{i = 1}^{n+1} y_{i} \right)^{-(n+1)}\ \lambda_{\bbR^{n+1}}\right)} \ar[d]^{H = 
\rsfsH_{n+1}} \\
*+[l]{(\bbR^{n}, \sqrt{n+1}\ \lambda_{\bbR^{n}}) =
(T, \tau)} \ar[r]_{\Psi = \rsfsC_{n}} &
*+[r]{(Y, \nu) =
(P_{n}, \mu_{P_{n}})}\\
}
\]

Note that as we are dealing with a manifold, albeit a trivial one ($P_{n}$), as the ultimate destination set,
it will be advantageous to read our measures thru the parameterisations, \emph{i.e.} only compute
$\eta_{a}$ and $\eta_{s}$.

For convenience's sake we will adopt the following notations in this section, which are also coherent with
those of subsection~\ref{meame}:
$\mu_{(\bbR_{+}^{\ast})^{n+1}} = \mu_{U_{n+1}}\vert_{(\bbR_{+}^{\ast})^{n+1}}$,
$\mu_{C_{n}} = \mu_{P_{n}}\vert_{C_{n}}$.

It is a trivial application of the Fubini-Lebesgues theorem to see that $\pi_{n+1}$ is not a metromorphism, and therefore neither is $\rsfsH_{n+1}$ (because $\rsfsT_{n+1}$ and $\rsfsC_{n}$ \emph{are}
metromorphisms), and thus lemma~\ref{elemmorph} can not be used in that context. however some part of it,
though not all (see \ref{oddsends}), can still be salvaged:

\begin{lemma}\label{homreg}
Let $\lambda$ be a complex measure on $\left( U_{n+1}, \Bor(U_{n+1})\right)$ such that
$\lambda \lll \mu_{U_{n+1}}$, then
$\left(\rsfsH_{n+1}\,{}_{\ast}\,\lambda\right) \lll \mu_{P_{n}}$.
\end{lemma}

\begin{proof}
Let $B \in \Bor(P_{n})$ such that $\mu_{P_{n}}\!<\!B\!> = 0$. By definition,
$\left(\rsfsH_{n+1}\,{}_{\ast}\,\lambda\right)\!<\!B\!> = \lambda(\rsfsH_{n+1}^{-1}\!<\!B\!>)$.

Since $\rsfsT_{n+1}$ is a metromorphism,
$\lambda(\rsfsH_{n+1}^{-1}\!<\!B\!>) =
(\rsfsT_{n+1}^{-1}\,{}_{\ast}\,\lambda)((\rsfsT_{n+1}^{-1} \circ \rsfsH_{n+1}^{-1})\!<\!B\!>)$.

However,
$(\rsfsT_{n+1}^{-1} \circ \rsfsH_{n+1}^{-1})\!<\!B\!> = (\pi_{n+1}^{-1} \circ \rsfsC_{n}^{-1})\!<\!B\!>$
because of the commutativity of the graph, and\linebreak
$(\pi_{n+1}^{-1} \circ \rsfsC_{n}^{-1})\!<\!B\!> = \pi_{n+1}^{-1}\!<\!\rsfsC_{n}^{-1}\!<\!B\!>\!> =
\rsfsC_{n}^{-1}\!<\!B\!> \times\ \bbR$,
this later by definition of $\pi_{n+1}$.

Therefore,
$\left(\rsfsH_{n+1}\,{}_{\ast}\,\lambda\right)\!<\!B\!> =
(\rsfsT_{n+1}^{-1}\,{}_{\ast}\,\lambda)(\rsfsC_{n}^{-1}\!<\!\}\!>\!\times\ \bbR)$.

We have assumed that $\lambda \lll \mu_{U_{n+1}}$. Since $\rsfsT_{n+1}$ is a
metromorphism, $(\rsfsT_{n+1}^{-1}\,{}_{\ast}\,\lambda) \lll  \lambda_{\bbR^{n+1}}$,
and therefore there exists $f \in\LSpace^{1}(\lambda_{\bbR^{n+1}})$ such that
$(\rsfsT_{n+1}^{-1}\,{}_{\ast}\,\lambda)(\rsfsC_{n}^{-1}\!<\!B\!>\!\times\ \bbR) =
\int_{\bbR^{n+1}} \bbone_{\rsfsC_{n}^{-1}\!<\!B\!> \times\ \bbR}(t)\ f(t)\ dt$.

Since $\rsfsC_{n}$ is a metromorphism, $\mu_{P_{n}}\!<\!B\!> = 0$ implies that
$\sqrt{n+1}\ \lambda_{\bbR^{n}}(\rsfsC_{n}^{-1}\!<\!B\!>) = 0$, which of course means that
$\lambda_{\bbR^{n}}(\rsfsC_{n}^{-1}\!<\!B\!>) = 0$. But then, the Fubini-Lebesgues theorem proves that\linebreak
$\lambda_{\bbR^{n+1}}(\rsfsC_{n}^{-1}\!<\!B\!> \times\ \bbR) = 0$. Integrating an integrable
function over a set of measure zero yielding a result of zero, we have proved that
$\left(\rsfsH_{n+1}\,{}_{\ast}\,\lambda\right)\!<\!B\!> = 0$.
\end{proof}

As well, the following simple fact is useful to remember:

\begin{lemma}\label{projreg}
Let $\zeta$ be a complex measure on $\left(\bbR^{n+1}, \Bor(\bbR^{n+1})\right)$ such that
$\zeta \lll \lambda_{\bbR^{n+1}}$, then
$\left(\pi_{n+1}\,{}_{\ast}\,\zeta\right) \lll \left( \sqrt{n+1}\ \lambda_{\bbR^{n}} \right)$ and
\[
\frac{d\left(\pi_{n+1}\,{}_{\ast}\,\zeta\right)}{d\left(\sqrt{n+1}\ \lambda_{\bbR^{n}}\right)}(x_{1}, \cdots, x_{n}) =
\frac{1}{\sqrt{n+1}} \int_{0}^{+\infty} \frac{d\zeta}{d\lambda_{\bbR^{n+1}}}(x_{1}, \cdots, x_{n}, t)\ dt
\]
\end{lemma}

\begin{proof}
This is a direct application of the Fubini-Lebesgues theorem.
\end{proof}

Finally, recall\footnote{Brief reminder: by induction; trivially  $\lambda_{1}(B_{1}) = 1$ and
$\lambda_{n+1}(B_{n+1}) = \int_{\bbR^{n+1}}\bbone(B_{n+1})\ d\lambda_{n+1}$ (with $\bbone$
being the indicator function) which by Fubini-Lebesgues is equal to $\int_{]0;1[}\ \lambda_{n}(t.B_{n})\ dt$, and
$\lambda_{n}(t.B_{n}) = t^{n}\ \lambda_{n}(B_{n})$.}
that $\lambda_{n}(B_{n}) = \frac{1}{n!}$ and\footnote{By definition,
$\mu_{C_{n}}(C_{n}) = \int A_{ \rsfsC_{n}}\bbone(B_{n})\ d\lambda_{n}$, and
$A_{ \rsfsC_{n}} = \sqrt{n+1}$.}
that $\mu_{C_{n}}(C_{n}) = \frac{\sqrt{n+1}}{n!}$.

\subsection{Example measures and their homogeneous transforms}
\indent

While the primary aim of these examples is to be illustrative, it is hoped that some of them will turn out to also be
useful in practice.

Let $\rsfsD_{(a_{1}, \ldots, a_{n+1})} : \left[B_{n} \rightarrow \bbR ;
(x_{1}, \ldots , x_{n}) \mapsto \left( \prod_{i = 1}^{n} x_{i}^{a_{i}} \right)
\left(1-\sum_{i=1}^{n} x_{i}\right)^{\alpha_{n+1}-1}\right]$, with
$a_{1} > 0, \cdots, a_{n+1} > 0$. Here, as usual,
$\Beta\negthinspace: \left[ \left( \bbR_{+}^{\ast} \right)^{n+1} \rightarrow  \bbR_{+}^{\ast}  ; 
(a_{1}, \ldots , a_{n+1}) \mapsto
\int_{B_{n}} \rsfsD_{(a_{1}, \ldots , a_{n+1})}(x_{1}, \ldots  , x_{n})\ dx_{1} \cdots dx_{n} \right]$
will denote the multivariate (Euler) Beta function\footnote{Recall that
$\Beta((a_{1}, \ldots , a_{n+1})) = \frac{\prod_{i = 1}^{n+1} \Gamma(a_{i})}{\Gamma(\sum_{i = 1}^{n+1} a_{i})}$,
the proof of which is by noticing that, as a consequence of the Fubini-Lebesgues theorem,
$\prod_{i = 1}^{n+1} \Gamma(a_{i}) =
\int_{0}^{+\infty} \cdots  \int_{0}^{+\infty} t_{1}^{a_{1}-1} \cdots t_{n+1}^{a_{n+1}-1}
\mathrm{e}^{-(t_{1}+ \cdots t_{n+1})} dt_{1} \cdots dt_{n+1}$,
by using the change of variables\linebreak
$[B_{n} \times \bbR_{+}^{\ast} \rightarrow (\bbR_{+}^{\ast})^{n+1},
((x_{1}, \ldots, x_{n}), x_{n+1}) \mapsto (t_{1}, \ldots , t_{n+1}) =
(x_{1}x_{n+1}, \ldots, x_{n} x_{n+1}, (1-x_{1} \cdots x_{n}) x_{n+1})]$
and by invoking Fubini-Lebesgues once again.} and
$\Gamma\negthinspace: \left[ \bbC-(-\bbN^{\ast}) \rightarrow \bbC; z \mapsto
\int_{0}^{+\infty} t^{z-1}\ \mathrm{e}^{-t}\ dt \right]$
will denote the Euler Gamma function.

Recall that the Dirichlet distribution on $B_{n}$ is the probability measure
$\Dirichlet(a_{1},\!\ldots \!,a_{n+1})\!=\!
\frac{\rsfsD_{(a_{1},\!\ldots\!,a_{n+1})}}{\Beta((a_{1},\!\ldots\!,a_{n+1}))} \lambda_{\bbR^{n}}$,
and therefore that $\rsfsC_{n}\,{}_{\ast}\,\Dirichlet(a_{1}, \ldots, a_{n+1})$ is also a probability measure on
$C_{n}$, whose expression, in barycentric coordinates, is
$\frac{1}{\sqrt{n+1}} \frac{y_{1}^{a_{1}} \cdots y_{n+1}^{a_{n+1}}}{\Beta((a_{1}, \ldots , a_{n+1})} \mu_{C_{n}}$.

\subsubsection{Dirac measures}
\indent

Dirac measures are used to model a perfect knowledge of the value of some physical quantity (or, less flatteringly,
when uncertainties are simply ignored).

Consider $M \in U_{n+1}$ and let $\delta_{M}$ be the Dirac
measure at $M$. Then trivially
$ \rsfsH_{n+1}\,{}_{\ast}\,\delta_{M} = \delta_{\rsfsH_{n+1}(M)}$.

Note that the Dirac measure at any other point on the half-line from $O$ to $M$ as chosen above will yield the same
transformed measure: the homogeneous transform is clearly not injective when considered on measures!

\subsubsection{Further examples of the non-injectivity}\label{noninj}
\indent

The non-injectivity of the homogeneous transform is of course not restricted to measures which are singular with
respect to the Lebesgues measure. Indeed, let $s > 1$ and consider
$k_{s}: \bbR_{+}^{\ast} \rightarrow \bbR_{+}, z \mapsto \frac{\kappa_{s}}{1+z^{s (n+1)}}$, with
$\kappa_{s} = \frac{(n+1)!}{\int_{0}^{+\infty} \frac{dz}{1+z^{s}}}$.
We now build $\theta_{s}$, a measure on $\Bor\left((\bbR_{+}^{\ast})^{n+1}\right)$ by
$\theta_{s} = k_{s}(\sum_{i = 1}^{n+1} y_{i})\ \lambda_{\bbR^{n+1}}$. Lemma~\ref{homreg} proves that
$(\rsfsH_{n+1}\,{}_{\ast}\,\theta_{s}) \lll \mu_{C_{n}}$; we will denote
$d(\rsfsH_{n+1}\,{}_{\ast}\,\theta_{s} ) / d\mu_{C_{n}}$ by $l_{s}$.

Note that the density of $\theta_{s}$ with respect to
$\mu_{U_{n+1}}$ is $(y_{1}, \ldots , y_{n+1}) \mapsto (\sum_{i = 1}^{n+1} y_{i})^{n+1} k_{s}(\sum_{i = 1}^{n+1} y_{i})$.
Using proposition~\ref{densimorph},  $(\rsfsT_{n+1}^{-1}\,{}_{\ast}\,\theta_{s}) \lll \lambda_{\bbR^{n+1}}$
and the density of $\rsfsT_{n+1}^{-1}\,{}_{\ast}\,\theta_{s}$ with respect to $\lambda_{\bbR^{n+1}}$ is
$\mathrm{e}^{(n+1)\:t}\ k_{s}(\mathrm{e}^{t})$.

Using lemma~\ref{projreg} we find that
$(\pi_{n+1}\,{}_{\ast}\,\rsfsT_{n+1}^{-1}\,{}_{\ast}\,\theta_{s}) \lll
\left( \sqrt{n+1}\ \lambda_{\bbR^{n}} \right)$
and that
\[\frac{d(\pi_{n+1}\,{}_{\ast}\,\rsfsT_{n+1}^{-1}\,{}_{\ast}\,\theta_{s})}
{d\left(\sqrt{n+1}\ \lambda_{\bbR^{n}} \right)} =
\frac{1}{\sqrt{n+1}} \int_{\bbR}
\mathrm{e}^{(n+1)\,t}\ k_{s}\!\left(\mathrm{e}^{t}\right)\ dt\]

The simple change of variables $u = \mathrm{e}^{(n+1)\,t}$ now lets us see that
$\rsfsH_{n+1}\,{}_{\ast}\,\theta_{s} = \frac{n!}{\sqrt{n+1}} \mu_{C_{n}}$.
In other words, all the $k_{s}$ are transformed by the homogeneous transform into the uniform probability on
$C_{n}$.

\subsubsection{Odds and ends}\label{oddsends}
\indent

Note that if $A \in \Bor((\bbR_{+}^{\ast})^{n+1})$, then $A \cap C_{n} \in \Bor(C_{n})$ and therefore
$\mu_{C_{n}}( A \cap C_{n} )$ is well defined, and we can thus build a measure $\rsfsR$ on
$ \Bor((\bbR_{+}^{\ast})^{n+1})$ by $\rsfsR(A) = \mu_{C_{n}}( A \cap C_{n} )$.
Clearly $\rsfsR \perp \mu_{(\bbR_{+}^{\ast})^{n+1}}$, yet
$\rsfsH_{n+1}\,{}_{\ast}\,\rsfsR = \mu_{C_{n}}$, the uniform density on $C_{n}$, and hence is
(trivially) absolutely continuous with respect to $\mu_{C_{n}}$,  and has the same homogeneous transform as the
$\theta_{s}$. This also shows that we can't expect better, from a general point of view, than what was stated in
lemma~\ref{homreg}.

We have seen examples of a measure which is singular with respect to $\mu_{(\bbR_{+}^{\ast})^{n+1}}$ and
whose homogeneous transform is singular with respect to  $\mu_{C_{n}}$, examples of measures which are
absolutely continuous with respect to  $\mu_{(\bbR_{+}^{\ast})^{n+1}}$ and
whose homogeneous transform are absolutely continuous with respect to  $\mu_{C_{n}}$ and finally now an example
of a measure which is   singular with respect to $\mu_{(\bbR_{+}^{\ast})^{n+1}}$ and
whose homogeneous transform is also absolutely continuous with respect to $\mu_{C_{n}}$. Is is possible to have
an example of the last possibility, namely a measure which would be absolutely continuous with respect to
$\mu_{(\bbR_{+}^{\ast})^{n+1}}$ but whose homogeneous transform would be singular with respect to
$\mu_{C_{n}}$? Actually, no, as we have shown in lemma~\ref{homreg}.

\subsubsection{Log-normal distributions}\label{lognormdist}
\indent

Let $\Sigma$ be a symmetric, positive definite, $(n+1)\times(n+1)$ matrix, and
$\bm{\mu} = (\bm{\mu}_{1}, \ldots , \bm{\mu}_{n+1}) \in \bbR^{n+1}$;
the log-normal distribution is the probability measure on
$(\bbR_{+}^{\ast})^{n+1}$ defined by
$\rsfsL_{\bm{\mu}, \Sigma} = f_{\bm{\mu}, \Sigma}\ \lambda_{\bbR^{n+1}}$,
with $f_{\bm{\mu}, \Sigma}\negthinspace:  \ (\bbR_{+}^{\ast})^{n+1} \rightarrow \bbR$ and
\[
f_{\bm{\mu}, \Sigma}(y_{1}, \ldots , y_{n+1}) =
\frac{\exp\left(-\frac{1}{2}
 \left[\ln(y_{1})-\bm{\mu}_{1}, \ldots , \ln( y_{n+1})-\bm{\mu}_{n+1} \right]
 \Sigma^{-1}
\left[\begin{array}{c}
\ln(y_{1})-\bm{\mu}_{1}\\
\vdots\\
\ln( y_{n+1})-\bm{\mu}_{n+1}
 \end{array}\right]\right)}
{\sqrt{\vert\det(\Sigma)\vert}(2\pi)^{\frac{n+1}{2}} \prod_{i = 1}^{n+1} y_{i}}
\]

Note that that $\Sigma$ being positive definite, we can define
\[
\rsfsV = \left([1,\ldots,1] \Sigma^{-1}
\left[\begin{array}{c}
 1\\
\vdots\\
1
\end{array}\right]
\right)^{-1}
\]
and we have $\rsfsV > 0$.
Furthermore, let
\begin{gather*}
\rsfsQ = \sqrt{\vert\det(\Sigma)\vert}(2\pi)^{\frac{n+1}{2}}
\left(1-\sum_{i = 1}^{n} x_{i} \right) \prod_{i = 1}^{n} x_{i}\\
\rsfsX\!=\![1,\!\ldots\!,1]\Sigma^{-1}\!
\left[\begin{array}{c}
\ln(x_{1})\!-\!\bm{\mu}_{1}\\
\vdots\\
\ln(x_{n})\!-\!\bm{\mu}_{n}\\
\ln(1-\sum_{i = 1}^{n} x_{i})-\bm{\mu}_{n+1}
\end{array}\right]\!=\!
\left[\ln(x_{1})\!-\!\bm{\mu}_{1},\!\ldots\!,\ln(x_{n})\!-\!\bm{\mu}_{n},
\ln\left(\!1\!-\!\sum_{i = 1}^{n} x_{i}\!\right)\!-\!\bm{\mu}_{n+1} \right]
\Sigma^{-1}
\left[\begin{array}{c}
1\\
\vdots\\
1
\end{array}\right]\\
\rsfsY =
\left[\ln(x_{1})-\bm{\mu}_{1}, \ldots , \ln(x_{n})-\bm{\mu}_{n},
\ln\left(1-\sum_{i = 1}^{n} x_{i}\right)-\bm{\mu}_{n+1} \right]
\Sigma^{-1}
\left[\begin{array}{c}
\ln(x_{1})-\bm{\mu}_{1}\\
\vdots\\
\ln(x_{n})-\bm{\mu}_{n}\\
\ln(1-\sum_{i = 1}^{n} x_{i})-\bm{\mu}_{n+1}
\end{array}\right]
 \end{gather*}

\begin{proposition}\label{lognormHT}
$\rsfsH_{n+1}\,{}_{\ast}\,\rsfsL_{\bm{\mu}, \Sigma} \lll \mu_{C_{n}}$ and
$d(\rsfsH_{n+1}\,{}_{\ast}\,\rsfsL_{\bm{\mu}, \Sigma} ) / d\mu_{C_{n}} = g_{\bm{\mu}, \Sigma}$, with
$g_{\bm{\mu}, \Sigma} \circ \rsfsC_{n}^{-1 }\negthinspace: B_{n} \rightarrow  \bbR$ and
\[
g_{\bm{\mu}, \Sigma} \circ \rsfsC_{n}^{-1 }(\ x_{1}, \ldots , x_{n}) =
\frac{\sqrt{2\pi \rsfsV}}{\rsfsQ\sqrt{n+1}}
\exp\left(-\frac{1}{2}\left[\rsfsY-\rsfsV\rsfsX^{2}\right]\right)
\]
\end{proposition}

\begin{proof}
We proceed almost exactly as in~\ref{noninj}. The first assertion is a direct result of lemma~\ref{homreg}.

Note that the density of $\rsfsL_{\bm{\mu}, \Sigma}$ with respect to $\mu_{U_{n+1}}$ is
$(y_{1}, \ldots , y_{n+1}) \mapsto f(y_{1}, \ldots , y_{n+1}) \left(\sum_{i = 1}^{n+1} y_{i} \right)^{n+1}$.
Using proposition~\ref{densimorph}, we find that
$\rsfsT_{n+1}^{-1}\,{}_{\ast}\,\rsfsL_{\bm{\mu}, \Sigma} \lll \lambda_{\bbR^{n+1}}$
and that the density of $\rsfsT_{n+1}^{-1}\,{}_{\ast}\,\rsfsL_{\bm{\mu}, \Sigma}$ with respect to
$\lambda_{\bbR^{n+1}}$ is
\linebreak$\mathrm{e}^{(n+1)\,t} f_{\bm{\mu}, \Sigma}\left(\mathrm{e}^{t}.x_{1},\ldots,\mathrm{e}^{t}.x_{n},
\mathrm{e}^{t}.\left(1-\sum_{i = 1}^{n} x_{i} \right)\right)$.

Using lemma~\ref{projreg} we find that
$(\pi_{n+1}\,{}_{\ast}\,\rsfsT_{n+1}^{-1}\,{}_{\ast}\,\rsfsL) \lll
\left(\sqrt{n+1}\ \lambda_{\bbR^{n}} \right)$
and that
\[\frac{d(\pi_{n+1}\,{}_{\ast}\,\rsfsT_{n+1}^{-1}\,{}_{\ast}\,\rsfsL_{\bm{\mu}, \Sigma})}
{d\left(\sqrt{n+1}\ \lambda_{\bbR^{n}} \right)} =
\frac{1}{\sqrt{n+1}} \int_{\bbR}
\mathrm{e}^{(n+1)\,t}
f_{\bm{\mu}, \Sigma}\left(\mathrm{e}^{t}.x_{1},\ldots,\mathrm{e}^{t}.x_{n},\mathrm{e}^{t}.
\left(1-\sum_{i = 1}^{n} x_{i} \right)\right)
dt\]

We note that
\begin{align*}
\mathrm{e}^{(n+1)\,t}
f_{\bm{\mu}, \Sigma}\left(\mathrm{e}^{t}.x_{1},\ldots,t.x_{n}, \mathrm{e}^{t}.\left(1-\sum_{i = 1}^{n} x_{i} \right)\right) &=
\frac{1}{\rsfsQ} \exp\left(-\frac{1}{2}\left[\frac{t^{2}}{\rsfsV}+
2 \rsfsX t+\rsfsY\right]\right)\\
&= \frac{\exp\left(-\frac{1}{2}\left[\rsfsY-\rsfsV\rsfsX^{2}\right]\right)}{\rsfsQ}
\exp\left(-\frac{1}{2\rsfsV}\left[t+\rsfsV\rsfsX\right]^{2}\right)
\end{align*}
and, classically, $\int_{\bbR} \exp\left(-\frac{1}{2\rsfsV}\left[t+\rsfsV\rsfsX\right]^{2}\right) dt =
\sqrt{2\pi\rsfsV}$.
\end{proof}

Note that in barycentric coordinates, the expressions for $\rsfsQ$, $\rsfsX$ and $\rsfsY$
are considerably simpler:
\begin{gather*}
\rsfsQ = \sqrt{\vert\det(\Sigma)\vert}(2\pi)^{\frac{n+1}{2}}
\prod_{i = 1}^{n+1} y_{i}\\
\rsfsX = [1,\ldots,1] \Sigma^{-1}
\left[\begin{array}{c}
\ln(y_{1})-\bm{\mu}_{1}\\
\vdots\\
\ln(y_{n+1})-\bm{\mu}_{n+1}
\end{array}\right] =
\left[\ln(y_{1})-\bm{\mu}_{1}, \ldots , \ln(y_{n+1})-\bm{\mu}_{n+1}\right]
\Sigma^{-1}
\left[\begin{array}{c}
1\\
\vdots\\
1
\end{array}\right]\\
\rsfsY =
\left[\ln(y_{1})-\bm{\mu}_{1}, \ldots , \ln(y_{n+1})-\bm{\mu}_{n+1}\right]
\Sigma^{-1}
\left[\begin{array}{c}
\ln(y_{1})-\bm{\mu}_{1}\\
\vdots\\
\ln(y_{n+1})-\bm{\mu}_{n+1}
\end{array}\right]
\end{gather*}

Furthermore, if $\Sigma$ is diagonal, with
\[
\Sigma =
\left[
\begin{array}{c}
\xymatrix{
\sigma_{1}^{2} \ar@{.}[drdr] & 0 \ar@{.}[r] \ar@{.}[dr] & 0 \ar@{.}[d]\\
0 \ar@{.}[d] \ar@{.}[dr] & & 0 \\
0 \ar@{.}[r] & 0 & \sigma_{n+1}^{2} }
\end{array}
\right]
\]
then
\[
\begin{array}{ccccccc}
\rsfsV = \frac{1}{\frac{1}{\sigma_{1}^{2}}+\cdots+\frac{1}{\sigma_{n+1}^{2}}} & , &
\rsfsQ = (2\pi)^{\frac{n+1}{2}}\prod_{i = 1}^{n+1} \frac{y_{i}}{\sigma_{i}} & , &
\rsfsX = \sum_{i = 1}^{n+1} \frac{\ln(y_{i})-\bm{\mu}_{i}}{\sigma_{i}^{2}} & , &
\rsfsY = \sum_{i = 1}^{n+1} \frac{\left(\ln(y_{i})-\bm{\mu}_{i}\right)^{2}}{\sigma_{i}^{2}}
\end{array}
\]

\subsubsection{Multivariate Gamma distributions}
\indent

Recall that a Gamma distribution on $\bbR_{+}^{\ast}$ is
defined as $\gamma_{a,b}\  \lambda_{\bbR}$, with
$\gamma_{a,b}(x) = \frac{1}{\Gamma(a)}\ b^{a} x^{a-1} \mathrm{e}^{-bx}$, $a > 0$, $b > 0$.

Let us define here a multivariate Gamma distribution $\rsfsG_{\alpha,\beta}$ on
$(\bbR_{+}^{\ast})^{n+1}$ , with $\alpha = (\alpha_{1}, \ldots, \alpha_{n+1})$,\linebreak
$\beta = (\beta_{1}, \ldots, \beta_{n+1})$, $\alpha_{i} > 0$, $\beta_{i} > 0$, as
$\rsfsG_{\alpha,\beta} = f_{\alpha, \beta}\ \lambda_{\bbR^{n+1}}$ with
$f_{\alpha, \beta}(y_{1},\!\ldots,\!y_{n+1}) = \gamma_{\alpha_{1}, \beta_{1}}(y_{1})\!\cdots\!
\gamma_{\alpha_{n+1}, \beta_{n+1}}(y_{n+1})$.

\begin{proposition}\label{mulvargaHT}
$\rsfsH_{n+1}\,{}_{\ast}\,\rsfsG_{\alpha, \beta} \lll \mu_{C_{n}}$ and
$d(\rsfsH_{n+1}\,{}_{\ast}\,\rsfsG_{\alpha, \beta} ) / d\mu_{C_{n}} = g_{\alpha, \beta}$, with
$g_{\alpha, \beta} \circ \rsfsC_{n}^{-1 }\negthinspace: B_{n} \rightarrow \bbR$ 
\[
g_{\alpha, \beta} \circ \rsfsC_{n}^{-1 }(x_{1}, \ldots , x_{n}) =
\frac{1}{\sqrt{n+1}} \frac{\rsfsD_{\alpha}(x_{1}, \ldots , x_{n})}{\Beta((\alpha_{1}, \ldots , \alpha_{n+1}))}
\frac{\beta_{1}^{\alpha_{1}} \cdots \beta_{n+1}^{\alpha_{n+1}}}
{\left[ \beta_{1} x_{1} + \cdots + \beta_{n} x_{n} +
\beta_{n+1} \left(1 - \sum_{i = 1}^{n} x_{i} \right) \right]^{\alpha_{1} + \cdots + \alpha_{n+1}}}
\]
\end{proposition}

\begin{proof}
We proceed exactly as in~\ref{lognormdist}. The first assertion is a direct result of lemma~\ref{homreg}.
Likewise,
\linebreak$g_{\alpha, \beta} \circ \rsfsC_{n}^{-1}\!=\!\frac{1}{\sqrt{n+1}}\!\int_{\bbR}
\mathrm{e}^{(\alpha_{1}\!+\!\cdots\!+\!\alpha_{n+1})\,t}
\frac{\beta_{1}^{\alpha_{1}} \cdots \beta_{n+1}^{\alpha_{n+1}}}{\Gamma(\alpha_{1}) \cdots \Gamma(\alpha_{n+1})}
x_{1}^{\alpha_{1}-1}\!\cdots\!x_{n}^{\alpha_{n}-1}\!\left(1\!-\!\sum_{i = 1}^{n} x_{i}\right)^{\alpha_{n+1}-1}\!
\mathrm{e}^{-\mathrm{e}^{t} \left[
\beta_{1} x_{1}+\!\cdots\!+\beta_{n} x_{n}+\beta_{n+1} \left(1\!-\!\sum_{i = 1}^{n} x_{i} \right)
\right]} dt$.

We now use the change of variables $[\bbR_{+}^{\ast} \rightarrow \bbR ; u \mapsto t = \ln(u)]$
and find that\linebreak
$g_{\alpha, \beta} \circ \rsfsC_{n}^{-1}\!=\!\frac{1}{\sqrt{n+1}}
\frac{(\beta_{1}^{\alpha_{1}} \cdots \beta_{n+1}^{\alpha_{n+1}})
[x_{1}^{\alpha_{1}-1} \cdots x_{n}^{\alpha_{n}-1} \left(1\!-\!\sum_{i = 1}^{n} x_{i}\right)^{\alpha_{n+1}}]}
{\Gamma(\alpha_{1}) \cdots \Gamma(\alpha_{n+1})}
\int_{\bbR_{+}^{\ast}} u^{\alpha_{1} + \cdots \alpha_{n+1} - 1}
\ \mathrm{e}^{-u \left[ \beta_{1} x_{1} + \cdots + \beta_{n} x_{n} +
\beta_{n+1} \left(1\!-\!\sum_{i = 1}^{n} x_{i}\right)\right]} du$.\linebreak
We identify the last integral as the (unilateral) Laplace transform
of $\left[ u \mapsto u^{\alpha_{1} + \cdots \alpha_{n+1} - 1} \right]$ for\linebreak
$p = \beta_{1} x_{1} + \cdots + \beta_{n} x_{n} + \beta_{n+1} \left(1 - \sum_{i = 1}^{n} x_{i} \right)$
and we therefore know (\cite{Churchill(OM)}) that\linebreak
$\int_{\bbR_{+}^{\ast}} u^{\alpha_{1} + \cdots \alpha_{n+1} - 1}
\ \mathrm{e}^{-u \left[ \beta_{1} x_{1} + \cdots + \beta_{n} x_{n} +
\beta_{n+1} \left(1 - \sum_{i = 1}^{n} x_{i} \right)\right]} du =
\frac{\Gamma(\alpha_{1} + \cdots + \alpha_{n+1})}
{\left[ \beta_{1} x_{1} + \cdots + \beta_{n} x_{n} +
\beta_{n+1} \left(1 - \sum_{i = 1}^{n} x_{i} \right) \right]^{\alpha_{1} + \cdots + \alpha_{n+1}}}$.
\end{proof}

\subsubsection{Multivariate Chi distributions}
\indent

Recall that a Chi distribution on $\bbR_{+}^{\ast}$ is
defined as $\chi_{k}\ \lambda_{\bbR}$ with
$\chi_{k}(x) = \frac{1}{\Gamma\left( \frac{k}{2} \right)}\ 2^{1-\frac{k}{2}}\ x^{k-1}\ \mathrm{e}^{-\frac{x^{2}}{2}}$, $k > 0$.
Let us define here a multivariate Chi distribution $\Chi_{\bm{k}}$ on
$(\bbR_{+}^{\ast})^{n+1}$ , with $\bm{k} = (k_{1}, \ldots, k_{n+1})$, as
$\Chi_{\bm{k}} = f_{\bm{k}}\ \lambda_{\bbR^{n+1}}$ with
\linebreak$f_{\bm{k}}(y_{1}, \ldots, y_{n+1}) = \chi_{k_{1}}(y_{1}) \cdots  \chi_{k_{n+1}}(y_{n+1})$.

\begin{proposition}\label{mulvarchiHT}
$\rsfsH_{n+1}\,{}_{\ast}\,\Chi_{\bm{k}} \lll \mu_{C_{n}}$ and
$d(\rsfsH_{n+1}\,{}_{\ast}\,\Chi_{\bm{k}} ) / d\mu_{C_{n}} = g_{\bm{k}}$, with
$g_{\bm{k}} \circ \rsfsC_{n}^{-1 }\negthinspace: B_{n} \rightarrow \bbR$ and
\[
g_{\bm{k}} \circ \rsfsC_{n}^{-1 }(x_{1},\!\ldots\!,x_{n})\!=\!
\frac{1}{\sqrt{n+1}}\ \frac{\Gamma \left(\frac{k_{1}+\!\cdots\!+k_{n+1}}{2^{n+1}}\right)}
{\Gamma\left(\frac{k_{1}}{2}\right)\!\cdots\!\Gamma\left(\frac{k_{n+1}}{2}\right)}
\ 2^{(n+1)-\left(k_{1}\!+\!\cdots\!+k_{n+1}\right) \frac{n}{2 (n+1)}}
\ \frac{x_{1}^{k_{1}-1}\!\cdots\!x_{n}^{k_{n}-1}\!\left(1\!-\!\sum_{i = 1}^{n} x_{i}\right)^{k_{n+1}-1}}
{\left[x_{1}\!\cdots\!x_{n}\!\left(1\!-\!\sum_{i = 1}^{n} x_{i}\right)\right]^{\frac{k_{1}+\!\cdots\!+k_{n+1}}{n+1}}}
\]
\end{proposition}

\begin{proof}
We once again proceed exactly as in~\ref{lognormdist}. We are led to evaluate
$\int_{\bbR}\mathrm{e}^{t (k_{1}+\!\cdots\!+k_{n+1})}
\mathrm{e}^{-\!\frac{1}{2}\!\left[x_{1}^{2}\!\cdots\!x_{n}^{2} \left(1\!-\!\sum_{i = 1}^{n} x_{i}\right)^{2} \right]
\mathrm{e}^{t 2 (n+1)}}$,
we use the change of variables $[\bbR_{+}^{\ast} \rightarrow \bbR ; u \mapsto t = \frac{1}{2 (n+1)} \ln(u)]$
and we identify the (unilateral) Laplace transform of
$\left[ u \mapsto u^{\frac{1}{2 (n+1)} \left( k_{1} + \cdots k_{n+1} \right)-1} \right]$ for
$p = \frac{1}{2} \left[ x_{1}^{2} \cdots x_{n}^{2}\ \left( 1- \sum_{i = 1}^{n} x_{i} \right)^{2} \right]$.
\end{proof}

\subsection{Further considerations}
\indent

For display purposes, or as a step in data classification, it is usually advantageous to lower the dimensionality.
There is of course an infinite number of ways of embedding $C_{n}$ into $\bbR^{n}$,
even among affine transforms, and there does not seem to be a universally agreed-upon ``best'' way to do so
for these tasks in the general case, except perhaps for the homogeneous transforms of bivariate and trivariate data.
When transforming a bivariate distribution, a convenient representation is to represent the density as a function
of the percentage of one of the variables (in essence, illustrating the density thru $\rsfsC_{n}$); an example
of such is shown\footnote{These figures were made using Matplotlib (\cite{Hunter(2007)}), a 2D graphics package
used for the Python programming language (Python Software Foundation, \url{https://www.python.org})}
in Figure~\ref{multi-gamma}.
\begin{figure}[p]
\centering
\includegraphics[scale=0.6]{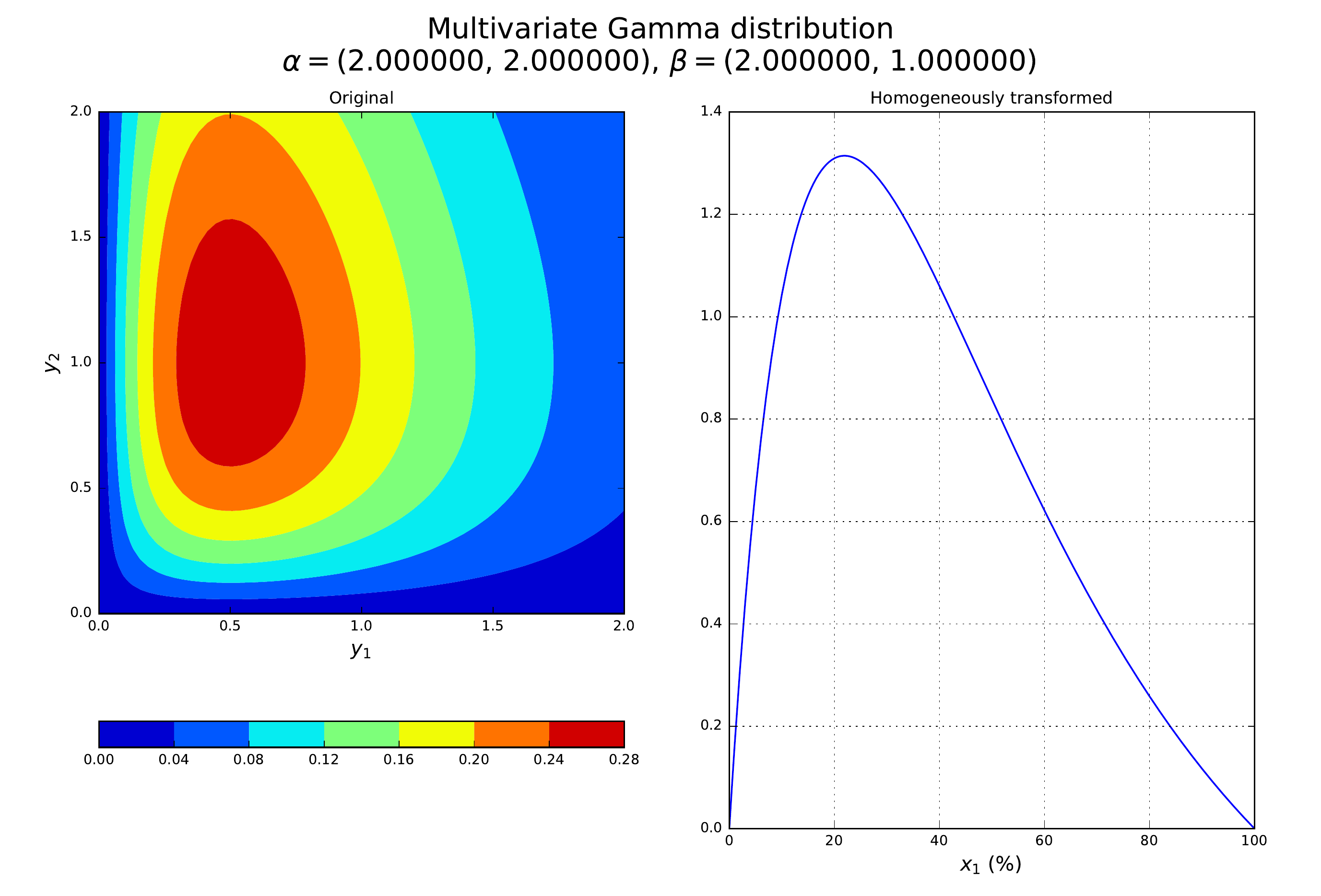}
\caption{A bivariate Gamma distribution and its homogeneous transform.}\label{multi-gamma}
\end{figure}
For trivariate data, especially if classification is intended, it is advantageous to use an isometric representation
of the homogeneous transform and this is shown in figure~\ref{lognormal}.
\begin{figure}[p]
\centering
\includegraphics[scale=0.6]{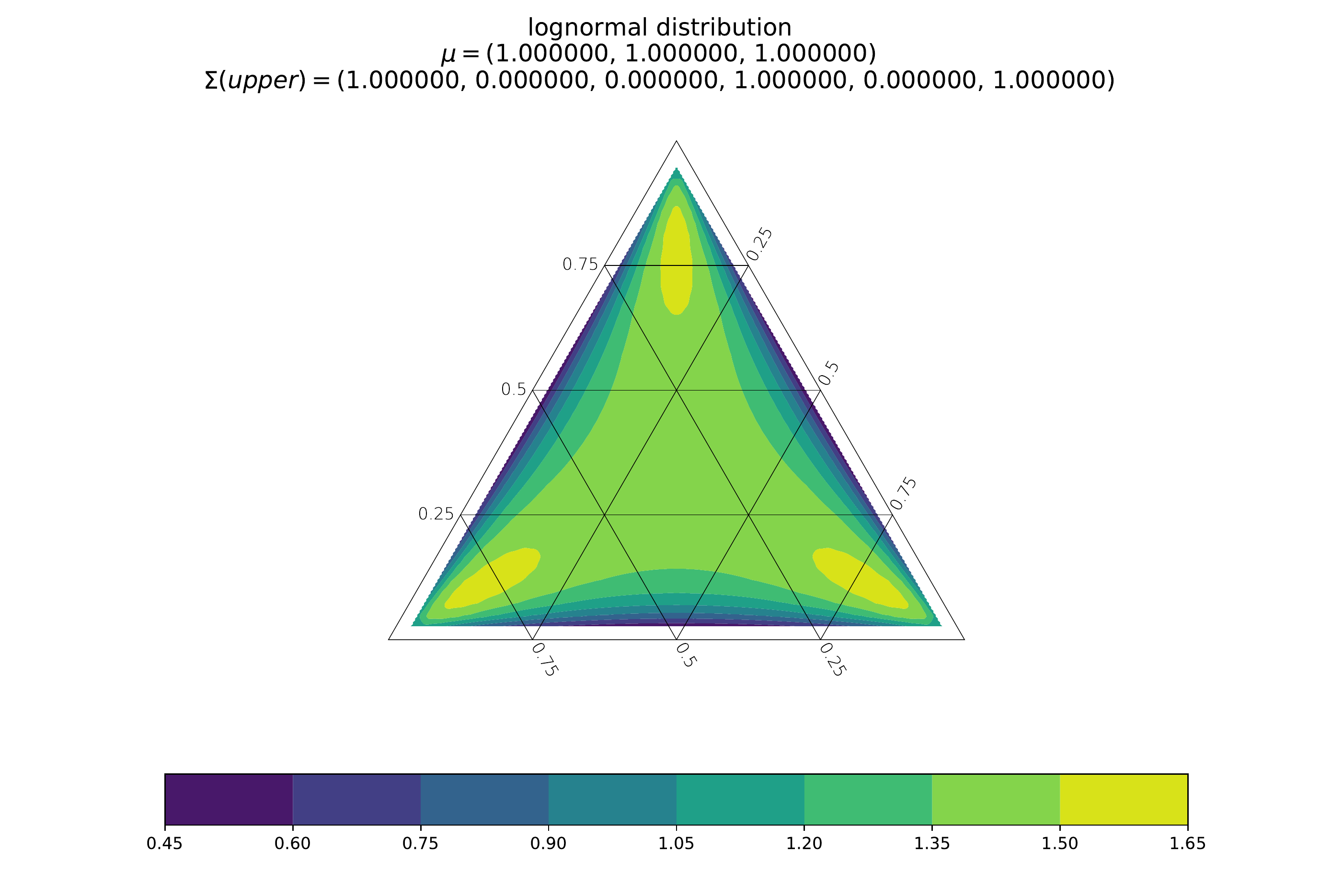}
\caption{The homogeneous transform of a trivariate log-normal distribution.}\label{lognormal}
\end{figure}

When using the measures showcased here to represent uncertainty about measures, one should bear in mind that
their behaviour is somewhat different than that of the common Gaussian distribution. In particular, save for
some select values of the parameters and levels, level curves of the densities do not enclose convex domains.

Finally, one should note that, when $\mu$ is null, the homogeneous transform of the log-normal distribution
presents striking symetries.

\pagebreak
\bibliography{MeasuresHT}
\bibliographystyle{kpmodHH}

 \end{document}